\documentclass[a4paper,12pt]{article}
\usepackage[latin1]{inputenc}
\usepackage{amsthm}
\usepackage{amsfonts}
\usepackage[dvips]{graphicx}
\usepackage{latexsym}
\usepackage{amsmath}
\usepackage{color}

\usepackage{amssymb}

\newtheorem{thm}{Theorem}
\newtheorem{defn}{Definition}
\newtheorem{prop}{Proposition}
\newtheorem{cor}{Corollary}
\newtheorem{oss}{Remark}

\newcommand{\e}{\epsilon}
\newcommand{\phie}{\varphi_{\epsilon}}
\newcommand{\phij}{\varphi_j}
\newcommand{\ue}{u_{\epsilon}}
\newcommand{\mue}{\mu_{\epsilon}}
\newcommand{\Fe}{F_{\epsilon}}
\newcommand{\Fie}{F_{1\epsilon}}
\newcommand{\me}{m_{\epsilon}}
\newcommand{\ze}{z_{\epsilon}}

\newcommand{\duavg}[1]{\left\langle{#1}\right\rangle}

\allowdisplaybreaks[4]

\begin{document}

\title{A diffuse interface model \\ for two-phase incompressible flows \\ with nonlocal interactions \\ and nonconstant mobility}

\author{
{Sergio Frigeri \thanks{{\color{black} Weierstrass Institute for Applied Analysis and Stochastics, Mohrenstr. 39, D-10117 Berlin,
Germany.
%Dipartimento di Matematica ``F. Enriques'',
%Universit\`{a} degli Studi di Milano, Milano I-20133, Italy.
E-mail: \textit{SergioPietro.Frigeri@wias-berlin.de}  The author
is supported by the FP7-IDEAS-ERC-StG
Grant \#256872 (EntroPhase)}}}
\and
{Maurizio Grasselli\thanks{Dipartimento di Matematica,
Politecnico di Milano,
Milano I-20133, Italy.
E-mail: \textit{mau\-ri\-zio.grasselli@polimi.it}}}
\and
{Elisabetta Rocca\thanks{{\color{black}
Weierstrass Institute for Applied Analysis and Stochastics, Mohrenstr. 39, D-10117 Berlin,
Germany.
%Dipartimento di Matematica ``F. Enriques'',
%Universit\`{a} degli Studi di Milano, Milano I-20133, Italy.
E-mail: \textit{Elisabetta.Rocca@wias-berlin.de} The author
is supported by the FP7-IDEAS-ERC-StG
Grant \#256872 (EntroPhase)}}}}

\maketitle

\begin{abstract}\noindent
We consider a diffuse interface model for incompressible isothermal mixtures of two immiscible fluids with matched constant densities.
This model consists of the Navier-Stokes system coupled with a convective nonlocal Cahn-Hilliard equation with non-constant mobility. We first prove the existence of a global weak solution in the case of non-degenerate mobilities and regular potentials of polynomial growth. Then we extend the result to degenerate mobilities and singular (e.g. logarithmic)  potentials. In the latter case we also establish the existence of the global attractor in dimension two. Using a similar technique, we show that there is a global attractor for the convective nonlocal Cahn-Hilliard equation with degenerate  mobility  and singular potential in dimension three.
\end{abstract}

\section{Introduction}\setcounter{equation}{0}

Model H is a diffuse interface model for incompressible isothermal two-phase flows which consists of
the Navier-Stokes equations for the (averaged) velocity $u$ nonlinearly coupled with a convective Cahn-Hilliard equation for the (relative) concentration difference $\varphi$
(cf., for instance, \cite{AMW,GPV,HMR,HH,JV,Kim2012,LMM}).
The resulting evolution system has been studied by several authors (see, e.g., \cite{A1,A2,B,CG,GG1,GG2,GG3,LS,S,ZWH,ZF} and
references therein,  cf. also \cite{ADT,Bos,KCR,GP} for models with shear dependent viscosity).
In the case of matched densities, setting the constant density equal to one, the system can be written as follows
\begin{align}
& u_t-\nu\Delta u+(u\cdot\nabla)u+\nabla\pi=\mu\nabla\varphi+h\qquad \hbox{in }\Omega\times(0,T)\label{Sy01}\\
& \mbox{div}(u)=0\qquad \hbox{in }\Omega\times(0,T)\label{Sy02}\\
& \varphi_t+u\cdot\nabla\varphi=\mbox{div}(m(\varphi)\nabla\mu)\qquad \hbox{in }\Omega\times(0,T)\label{Sy03} \\
& \mu=-\sigma\Delta \varphi+\frac{1}{\sigma}F^\prime(\varphi)\qquad \hbox{in }\Omega\times(0,T), \label{Sy04}
\end{align}
where $\Omega\subset \mathbb{R}^d,$ $d=2,3,$ is a bounded domain. Here  $\nu>0$ is the viscosity (supposed to be constant for simplicity), $\pi$ is the pressure, $h$ is a given (non-gradient) external force, $m$ is the so-called mobility and $\sigma>0$ is related to the (diffuse) interface thickness.

A more realistic version of the Cahn-Hilliard equation is characterized by a (spatially) nonlocal free energy.
The physical relevance of nonlocal interactions
was already pointed out in the pioneering paper \cite{Ro} (see also \cite[4.2]{Em}
and references therein). Though
isothermal and nonisothermal models containing nonlocal terms have only recently
been studied from the analytical viewpoint (cf., e.g., \cite{BH1, CKRS, GZ, GL1, GL2, GLM, KRS2} and their references).
The difference between local and nonlocal models consists in the choice
of the  interaction potential. The nonlocal contribution to the free energy has typically
the form $\int_\Omega J(x,y)\,|\varphi(x) - \varphi(y)|^2\,dy$ with a given
symmetric kernel $J$ defined on  $\Omega\times \Omega$; its  local Ginzburg-Landau counterpart
has the form $(\sigma/2)|\nabla\varphi(x)|^2$
with a positive parameter $\sigma$. The latter can be obtained as a formal
limit as ${\color{black}\zeta}\to \infty$ from the nonlocal one with the choice $J(x,y)
= \zeta^{d+2} J(|\zeta(x-y)|^2)$, where $J$ is a nonnegative function with
support in $[0,1]$. This follows from the formula formally deduced in \cite{KRS}
\begin{align}\nonumber
\int_\Omega \zeta^{d+2} J(|\zeta(x-y)|^2)\,|\varphi(x) - \varphi(y)|^2\,dy
&= \int_{\Omega_\zeta(x)} J(|z|^2) \left|\frac{\varphi\left(x+ \frac{z}{\zeta}\right)
- \varphi(x)}{\frac{1}{\zeta}}\right|^2\,dz\\ \nonumber
\stackrel{\zeta\to\infty}{\longrightarrow}&
\int_{\mathbb{R}^d} J(|z|^2)\duavg{\nabla \varphi(x), z}^2\,dz
= \frac{\sigma}{2} |\nabla \varphi(x)|^2
\end{align}
for a sufficiently regular $\varphi$, where
$\sigma = 2/d\int_{\mathbb{R}^d} J(|z|^2)|z|^2\,dz$ and $\Omega_\zeta(x) =
\zeta(\Omega - x)$. Here we have used the identity
$\int_{\mathbb{R}^d} J(|z|^2)\duavg{e,z}^2\,dz = 1/d\,\int_{\mathbb{R}^d}
J(|z|^2)|z|^2\,dz$ for every unit vector $e \in\mathbb{R}^d$.
As a consequence, the Cahn-Hilliard equation \eqref{Sy03}-\eqref{Sy04} can be viewed as
an approximation of the nonlocal one.

Nonlocal interactions have been taken into account in a series of recent papers (see \cite{CFG,FGG,FG1,FG2,FGK})
where a modification of the model H with matched densities has been considered and analyzed .
More precisely, a system of the following form has been considered (cf. \cite{CFG})
\begin{align}
& \varphi_t+u\cdot\nabla\varphi=\mbox{div}(m(\varphi)\nabla\mu)\qquad \hbox{in }\Omega\times(0,T)\label{sy1}\\
& \mu=a\varphi-J\ast\varphi+F^\prime(\varphi)\qquad \hbox{in }\Omega\times(0,T)\label{sy2}\\
& u_t-\nu\Delta u+(u\cdot\nabla)u+\nabla\pi=\mu\nabla\varphi+h\qquad \hbox{in }\Omega\times(0,T)\label{sy3}\\
& \mbox{div}(u)=0\qquad \hbox{in }\Omega\times(0,T)\label{sy4}\\
& \frac{\partial\mu}{\partial n}=0,\quad u=0 \qquad\mbox{on }
\partial\Omega\times (0,T)\label{sy5}\\
& u(0)=u_0,\quad\varphi(0)=\varphi_0 \qquad\mbox{in }\Omega,
\label{sy6}
\end{align}
where $n$ stands for the outward normal to $\partial\Omega$, while $u_0$ and $\varphi_0$ are given initial conditions. Here
the interaction kernel $J:\mathbb{R}^d \to \mathbb{R}$ is an even function and $a(x) = \displaystyle\int_\Omega J(x-y)dy$.

Nonlocal system \eqref{sy1}-\eqref{sy4} is more challenging with respect to \eqref{Sy01}-\eqref{Sy04}, even in dimension two. One of
the reasons is that $\varphi$ has a poorer regularity and this influences the treatment of the Navier-Stokes system
through the so-called Korteweg force $\mu\nabla\varphi$. {\color{black} For instance, uniqueness in dimension two is an open issue
under sufficiently general conditions which ensures the existence of a weak solution (see \cite[Remark~8]{CFG},  cf. also \cite{FGG}).
Due to this   difficulty, only the constant mobility case has been considered so far (though viscosity depending
on $\varphi$ has been handled).
Let us briefly recall the main existing results for system \eqref{sy1}--\eqref{sy4} with $m$ constant.

In \cite{CFG} the authors proved the existence of global dissipative weak
solutions in 2D and 3D, and study some regularity properties, for the case of regular potentials of arbitrary polynomial growth.
For such potentials the longterm behavior of weak solutions was analyzed in \cite{FG1}. More precisely, the existence of the global attractor in 2D and
of the trajectory attractor in 3D were established. In \cite{FG2} the previous results of \cite{CFG,FG1} were extended to the case of singular potentials. The existence of (unique) strong solutions in 2D for regular potentials with arbitrary polynomial growth were obtained in \cite{FGK}.
There, in addition, the regularity of the global attractor and the convergence of weak solutions to single equilibria were shown.

Uniqueness of weak solutions in 2D has been demonstrated only recently for both regular and singular potentials (see \cite{FGG}) .
In the same paper, further results have been proven. For instance, the existence of strong solutions and the weak-strong uniqueness for the case of nonconstant (i.e., $\varphi-$dependent) viscosity and regular potentials, the existence of exponential attractors (regular potentials), and the connectedness of the global attractor in the case of constant viscosity.

 On the other hand, despite the variety of results with $m$ constant,} in the rigorous derivation of the nonlocal Cahn-Hilliard equation done in \cite{GL1} the mobility depends on $\varphi$ and degenerates at the pure phases, {\color{black} while the potential is of logarithmic type. This motivates the main goal of this contribution, namely, the analysis of the so-called nonlocal Cahn-Hilliard-Navier-Stokes system in the case of degenerate mobility and singular potential.

{\color{black} The local Cahn-Hilliard equation with degenerate mobility {\color{black} (i.e., system (\ref{sy1}--\ref{sy6}) with $u=0$)} was considered in the seminal paper \cite{EG}, where
the authors established the existence of a weak solution (cf. also \cite{LMS,ANC} and references therein, for nondegenerate mobility see \cite{BB,Sc})}. This result was then extended to the standard {\color{black}(local)} Cahn-Hilliard-Navier-Stokes system in \cite{B}. The nonlocal Cahn-Hilliard equation with degenerate mobility and logarithmic potential was rigorously justified and analyzed in \cite{GL1} (see also \cite{GLM} and references therein). In particular, in the case of periodic boundary conditions, an existence and uniqueness result was proven in \cite{GL2}. Then a more general case was considered in \cite{GZ}. The convergence to single equilibria was recently studied in \cite{LP,LP2} (cf. also \cite{GG4} for further results).

{\color{black} Inspired by} the strategy devised in \cite{EG}, we first analyze the nonlocal system by } taking a non-degenerate mobility $m$ and a regular potential $F$ with polynomial growth. We prove the existence of a global weak solution which satisfies an energy inequality (equality if $d=2$). This result extends \cite{FG1} and allows us to construct a rigorous approximation of the case where $m$ is degenerate and $F$ is singular (e.g. logarithmic). Therefore we can pass to the limit and obtain a similar result for the latter case. In addition, since the energy identity holds in two dimensions, we can construct a {\color{black} generalized} semiflow which possesses a global attractor by using Ball's method (see \cite{Ba}). {\color{black} In the above mentioned recent contribution \cite{FGG},  uniqueness of weak solutions to system \eqref{sy1}-\eqref{sy4}
with degenerate mobility and singular potential in 2D has also been proven (constant viscosity). Hence, we can say that our semiflow is indeed a semigroup and the global attractor is connected. Regarding the 3D case, the validity of a suitable energy
  inequality (cf. \eqref{energineq}) allows to generalize the results on the trajectory attractors (cf.~\cite{FG1,FG2}) to system \eqref{sy1}-\eqref{sy4} with degenerate mobility and singular potential. By means of Ball's approach, we can also show that the convective nonlocal Cahn-Hilliard equation with degenerate mobility and singular potential possesses a (connected) global attractor. In this case uniqueness can be proven even in dimension three.} Note that this result entails, in particular, that the nonlocal Cahn-Hilliard equation which has been obtained as hydrodynamic limit in \cite{GL1} has indeed a global attractor.

  {\color{black} We point out that uniqueness of solutions is still an open issue for the local Cahn-Hilliard equation analyzed in \cite{EG}. This is one of the main advantages of the nonlocal versus the local. We remind that uniqueness of weak solutions and continuous dependence estimates are fundamental starting points, for example, in view of the study of related optimal control problems. This issue will be the subject of  forthcoming contributions.}

Let us notice here that the main difficulty encountered while dealing with the degenerate mobility case
is that the gradient of the chemical potential $\mu$ in \eqref{sy2} can no longer be controlled
in any $L^p$ space. Hence, in order to get an existence result a suitable notion of weak solution \color{black} needs \color{black}
to be introduced {\color{black} (cf. \cite{EG} for the local Cahn-Hilliard equation)}. More precisely, in this new formulation the gradient of $\mu$ does not appear anymore (cf. Definition~\ref{wsdeg} in Section~\ref{sec:deg}). It worth observing that, in the present case, our main Thm.~\ref{Theor2} does not require the (conserved) mean value of the order parameter $\varphi$  to be strictly in between $-1$ and $1$, but $|\int_\Omega\varphi_0|\leq |\Omega|$
suffices.  Thus the model allows pure phase solutions for all $t\geq 0$. This is not possible in the case of constant or {\em strongly degenerate} mobility (cf. Remarks~\ref{pure} and \ref{purebis} for further comments on this topic).

We conclude by observing that it would be particularly interesting {\color{black} (albeit nontrivial)} to extend the present and previous results on nonlocal Cahn-Hilliard-Navier-Stokes systems to {\color{black} a model with} unmatched densities (for the local {\color{black} case} see \cite{ADG1,ADG2} and references therein)  or to the compressible case (cf. \cite{AbFei} {\color{black} for the local Cahn-Hilliard equation}).

The plan of the paper goes as follows. In Section~\ref{sec:notation} we introduce the notation and the functional setting. Section~\ref{sec:nondeg} is devoted to prove the existence of a global weak solution satisfying a suitable energy inequality (equality in the 2D case) when the mobility is non-degenerate and $F$ is regular. In Section~\ref{sec:deg}, using a convenient approximation scheme, we extend the proven result to the case of degenerate  mobility and a singular $F$ (e.g. of logarithmic type). Some regularity issues for $\varphi$ and $\mu$ are also discussed.
Section~\ref{sec:long} is devoted to the existence of global attractor in the two dimensional case. Finally, in Section~\ref{sec:convCahnHill},  we consider the convective nonlocal Cahn-Hilliard equation with degenerate mobility. We deduce the existence of a global weak solution from the previous result for the coupled system. Then, even in dimension three, we establish the uniqueness of weak solutions as well as the existence of a (connected) global attractor {\color{black} under rather general assumptions {\color{black} on $F$, $J$ and $m$.}
\color{black}

\section{Notation and functional setting}\setcounter{equation}{0}
\label{sec:notation}

We set $H:=L^2(\Omega)$ and $V:=H^1(\Omega)$, where $\Omega$ is supposed to have
a sufficiently smooth boundary (say, e.g., of class $C^{1,1}$).
If $X$ is a (real) Banach space, $X'$ will denote its dual.
For every $f\in V'$ we denote by $\overline{f}$ the average of $f$
over $\Omega$, i.e., $\overline{f}:=|\Omega|^{-1}\langle
f,1\rangle$. Here $|\Omega|$ is the Lebesgue measure of $\Omega$.
Let us introduce also the spaces $V_0:=\{v\in
V:\overline{v}=0\}$, $ V_0':=\{f\in V':\overline{f}=0\}$ and the
operator $A:V\to V'$, $A\in\mathcal{L}(V,V')$ defined by
$$\langle Au,v\rangle:=\int_{\Omega}\nabla u\cdot\nabla v\qquad\forall u,v\in V.$$
We recall that $A$ maps $V$ onto $V_0'$ and the restriction of $A$
to $V_0$ maps $V_0$ onto $V_0'$ isomorphically. Let us denote by
$\mathcal{N}:V_0'\to V_0$ the inverse map defined by
$$A\mathcal{N}f=f,\quad\forall f\in V_0'\qquad\mbox{and}\qquad\mathcal{N}Au=u,\quad\forall u\in V_0.$$
As is well known, for every $f\in V_0'$, $\mathcal{N}f$ is the
unique solution with zero mean value of the Neumann problem
\begin{equation*}
\left\{\begin{array}{ll}
-\Delta u=f,\qquad\mbox{in }\Omega\\
\frac{\partial u}{\partial n}=0,\qquad\mbox{on }\partial\Omega.
\end{array}\right.
\end{equation*}
Furthermore, the following relations hold
\begin{align}
&\langle Au,\mathcal{N}f\rangle=\langle f,u\rangle,\qquad\forall u\in V,\quad\forall f\in V_0',\nonumber\\
&\langle f,\mathcal{N}g\rangle=\langle
g,\mathcal{N}f\rangle=\int_{\Omega}\nabla(\mathcal{N}f)
\cdot\nabla(\mathcal{N}g),\qquad\forall f,g\in V_0'.\nonumber
\end{align}
Recall that $A$ can be also viewed as an unbounded operator
$A:D(A)\subset H\to H$ from the domain $D(A)=\{\phi\in
H^2(\Omega):\partial\phi/\partial n=0\mbox{ on }\partial\Omega\}$
onto $H$ and that $\mathcal{N}$ can also be viewed as a
self-adjoint compact operator $\mathcal{N}=A^{-1}:H\to H$ in $H$.
Hence, the fractional powers $A^r$ and $\mathcal{N}^s$, for
$r,s\geq 0$, can be defined through classical spectral theory.

We shall repeatedly need the standard Hilbert spaces for the Navier-Stokes
equations with no-slip boundary condition (see, e.g., \cite{T})
$$
G_{div}:=\overline{\{u\in
C^\infty_0(\Omega)^d:\mbox{
div}(u)=0\}}^{L^2(\Omega)^d},\quad
V_{div}:=\{u\in H_0^1(\Omega)^d:\mbox{ div}(u)=0\}.
$$
We denote by $\|\cdot\|$ and $(\cdot,\cdot)$ the norm and the
scalar product on both $H$ and $G_{div}$, respectively. Instead, $V_{div}$ is
endowed with the scalar product
$$(u,v)_{V_{div}}=(\nabla u,\nabla v),\qquad\forall u,v\in V_{div}.$$

In the proof of Theorem \ref{thm} we shall introduce the family of the eigenfunctions
of the Stokes operator $S$
with no-slip boundary condition. We recall that $S:D(S)\subset G_{div}\to G_{div}$ is defined as
$S:=-P\Delta$  with domain $D(S)=H^2(\Omega)^d\cap V_{div}$,
where $P:L^2(\Omega)^d\to G_{div}$ is the Leray projector. Notice that we have
$$(Su,v)=(u,v)_{V_{div}}=(\nabla u,\nabla v),\qquad\forall u\in D(S),\quad\forall v\in V_{div},$$
and $S^{-1}:G_{div}\to G_{div}$ is a self-adjoint compact operator in $G_{div}$.
Thus, according with classical results,
$S$ possesses a sequence of eigenvalues $\{\lambda_j\}$ with $0<\lambda_1\leq\lambda_2\leq\cdots$ and $\lambda_j\to\infty$,
and a family $\{w_j\}\subset D(S)$ of eigenfunctions which is orthonormal in $G_{div}$.
Let us also recall the Poincar\'{e} inequality
$$\lambda_1\Vert u\Vert^2\leq\Vert\nabla u\Vert^2,\qquad\forall u\in V_{div}.$$

%where $\lambda_1$ is the first eigenvalue of $S$ (cf \cite{T}).

The trilinear form $b$ which appears in the weak formulation of the
Navier-Stokes equations is defined as follows
$$b(u,v,w)=\int_{\Omega}(u\cdot\nabla)v\cdot w,\qquad\forall u,v,w\in V_{div}.$$
We recall that we have
\begin{equation}
b(u,w,v)=-b(u,v,w),\qquad\forall u,v,w\in V_{div}.\nonumber
\end{equation}
For  the basic estimates satisfied by the trilinear form $b$ the reader is referred to, e.g., \cite{T}.

 Finally, if $X$ is a (real) Banach space and $\tau\in\mathbb{R}$, we shall denote by $L^p_{tb}(\tau,\infty;X)$, $1\leq p<\infty$,
the space of functions $f\in L^p_{loc}([\tau,\infty);X)$ that are translation bounded in $L^p_{loc}([\tau,\infty);X)$,
i.e. such that
\begin{align}
&\Vert f\Vert_{L^p_{tb}(\tau,\infty;X)}^p:=\sup_{t\geq\tau}\int_t^{t+1}\Vert f(s)\Vert_X^p ds<\infty.
\end{align}

\section{Non-degenerate mobility}\setcounter{equation}{0}
\label{sec:nondeg}

Let us first consider the case where the mobility $m$ does not degenerate, i.e. $m$ satisfies the following assumption
\begin{description}
\item[(H1)] $m\in C^{0,1}_{loc}(\mathbb{R})$
and there exist $m_1,m_2>0$ such that
\begin{align}
m_1\leq m(s)\leq m_2,\qquad\forall s\in\mathbb{R}.\nonumber
\end{align}
\end{description}
The other assumptions we need are the ones on the kernel $J$, on
the potential $F$ and the forcing term $h$ which are the same as
in \cite{CFG}
\begin{description}
\item[(H2)] $J(\cdot-x)\in W^{1,1}(\Omega)$ for almost any
$x\in\Omega$ and satisfies
 \begin{align}
 & J(x)=J(-x),\qquad a(x) := \displaystyle
\int_{\Omega}J(x-y)dy \geq 0,\quad\mbox{ a.e. }
x\in\Omega,\nonumber
\end{align}
\begin{align}
&%\hspace{18mm}
a^\ast:=\sup_{x\in\Omega}\int_\Omega|J(x-y)|dy<\infty, \qquad
b:=\sup_{x\in\Omega}\int_\Omega|\nabla J(x-y)|dy<\infty.\nonumber
\end{align}

\item[(H3)] $F\in C^{2,1}_{loc}(\mathbb{R})$ and there exists $c_0>0$
    such that
            $$F^{\prime\prime}(s)+a(x)\geq c_0,\qquad\forall s\in\mathbb{R},\quad\mbox{a.e. }x\in\Omega.$$
\item[(H4)] %$F\in C^2(\mathbb{R})$ and there exist $c_1>0$,
    %$c_2>0$ and $q>0$ such that
    There exist $c_1>(a^\ast-a_\ast)/2$, where
   % $$M^\ast:=\sup_{x\in\Omega}\int_\Omega|J(x-y)|dy,\qquad
   $$a_\ast:=\inf_{x\in\Omega}\int_\Omega J(x-y)dy,$$
     and $c_2\in\mathbb{R}$ such that
            $$F(s)\geq c_1 s^2 - c_2,
            \qquad\forall s\in\mathbb{R}.$$

\item[(H5)] There exist $c_3>0$, $c_4\geq0$ and $r\in(1,2]$
    such that
            $$|F^\prime(s)|^r\leq c_3|F(s)|+c_4,\qquad
            \forall s\in\mathbb{R}.$$
\item[(H6)] $h\in L^2(0,T;V_{div}')\quad$ for all $T>0$.
\end{description}
Some further regularity properties of the weak solution can be established
by using the following assumption
\begin{description}
\item[(H7)] $F\in C^{2,1}_{loc}(\mathbb{R})$ and there exist $c_5>0$, $c_6>0$ and $p>2$
    such that
            $$F^{\prime\prime}(s)+a(x)\geq c_5\vert s\vert^{p-2} - c_6,
            \qquad\forall s\in\mathbb{R},\quad\mbox{a.e. }x\in\Omega.$$

\end{description}

\begin{oss}{\upshape
A well known example of potential $F$ and
kernel $J$ which satisfy the above conditions is given by
$F(s)=(s^2-1)^2$ and
\color{black} $J(x)=j_3|x|^{-1}$, \color{black} if $d=3$, and
$J(x)=-j_2\log|x|$, if $d=2$,
%\begin{align}
%&F(s)=(1-s^2)^2,\nonumber
%\end{align}
%and
%\begin{align}
%&J(x)=j_d|x|^{2-d},\qquad\mbox{if }\: d=3,\nonumber\\
%&J(x)=-j_2\log|x|,\qquad\mbox{if }\: d=2,\nonumber
%\end{align}
where $j_2$ and $j_3$ are positive constants.}
\end{oss}

Before stating the main result of this section, let us recall the definition of weak solution to
system \eqref{sy1}--\eqref{sy6}.
\begin{defn}
\label{wfdfn}
Let $u_0\in G_{div}$, $\varphi_0\in H$ such that $F(\varphi_0)\in
L^1(\Omega)$, and $0<T<\infty$ be given. Then, a couple
$[u,\varphi]$ is a weak solution to \eqref{sy1}--\eqref{sy6} on
$[0,T]$ if
\begin{align}
&u\in L^{\infty}(0,T;G_{div})\cap L^2(0,T;V_{div}),\qquad\varphi \in L^\infty(0,T;H)\cap L^2(0,T;V),
\label{propreg1}\\
&u_t\in L^{4/3}(0,T;V_{div}'),\qquad\varphi_t\in L^{4/3}(0,T;V'),\qquad\mbox{if }\: d=3,\label{propreg2}\\
&u_t\in L^{2-\gamma}(0,T;V_{div}'),\qquad\varphi_t\in L^{2-\delta}(0,T;V')
\qquad\gamma,\delta\in(0,1),\qquad\mbox{if }\: d=2,\label{propreg3}\\
%&\varphi_t\in L^{2-\delta}(0,T;V'),\qquad\delta\in(0,1)\qquad d=2
%\quad\mbox{ or }
 %\quad d=3 \mbox{ and } q\geq 1/2,
% \label{propreg4}
&\mu:=a\varphi-J\ast\varphi+F'(\varphi)\in L^2(0,T;V),\label{propreg7}
\end{align}
and the following variational formulation is satisfied for almost any $t\in(0,T)$
\begin{align}
&\langle\varphi_t,\psi\rangle+(m(\varphi)\nabla\mu,\nabla\psi)=(u\varphi,\nabla\psi),\qquad\forall\psi\in V,
\label{wdefnd1}\\
&\langle u_t,v\rangle+\nu(\nabla u,\nabla v)+b(u,u,v)=-(\varphi\nabla\mu,v)+\langle h,v\rangle,
\qquad\forall v\in V_{div},\label{wdefnd2}
\end{align}
together with the initial conditions \eqref{sy6}.
\end{defn}
\begin{oss}
\label{inconds}
{\upshape
It is easy to see that
$u\in C_w([0,T];G_{div})$ and $\varphi\in C_w([0,T];H)$.
Hence, the initial conditions \eqref{sy6} make sense.
}
\end{oss}

\begin{thm}
\label{thm} Let $u_0\in G_{div}$, $\varphi_0\in H$ such that
$F(\varphi_0)\in L^1(\Omega)$, and suppose that (H1)-(H6) are
satisfied. Then, for every given $T>0$, there exists a weak
solution $[u,\varphi]$ to \eqref{sy1}--\eqref{sy6}
 in the sense of Definition \ref{wfdfn} \color{black}
%such that
%\begin{align}
%&u\in L^{\infty}(0,T;G_{div})\cap L^2(0,T;V_{div}),\qquad\varphi \in L^\infty(0,T;H)\cap L^2(0,T;V),
%\label{propreg1}\\
%&u_t\in L^{4/3}(0,T;V_{div}'),\qquad\varphi_t\in L^{4/3}(0,T;V'),\qquad d=3,\label{propreg2}\\
%&u_t\in L^{2-\gamma}(0,T;V_{div}'),\qquad\varphi_t\in L^{2-\delta}(0,T;V')
%\qquad\gamma,\delta\in(0,1),\qquad d=2,\label{propreg3}
%%&\varphi_t\in L^{2-\delta}(0,T;V'),\qquad\delta\in(0,1)\qquad d=2
%\quad\mbox{ or }
 %\quad d=3 \mbox{ and } q\geq 1/2,
% \label{propreg4}
%\end{align}
%and
satisfying the energy inequality
\begin{equation}
\mathcal{E}(u(t),\varphi(t)) +\int_0^t\Big(\nu\|\nabla u \|^2+\|\sqrt{m(\varphi)}\nabla\mu \|^2\Big)d\tau
\leq\mathcal{E}(u_0,\varphi_0)+\int_0^t\langle h(\tau),u \rangle d\tau,\label{ei}
\end{equation}
for every $t>0$,
where we have set
\begin{align}
\mathcal{E}(u(t),\varphi(t))=\frac{1}{2}\|u(t)\|^2+\frac{1}{4}
\int_{\Omega}\int_{\Omega}J(x-y)(\varphi(x,t)-\varphi(y,t))^2
dxdy+\int_{\Omega}F(\varphi(t)).\label{endef}
\end{align}
Furthermore, assume that
assumption (H4) is replaced by (H7). Then, for every $T>0$ there
exists a weak solution $[u,\varphi]$ to \eqref{sy1}--\eqref{sy6}
on $[0,T]$ corresponding to $[u_0,\varphi_0]$  (in the sense of Definition \ref{wfdfn}) \color{black} satisfying
\eqref{propreg1}, \eqref{propreg2} and also
\begin{align}
&\varphi \in L^\infty(0,T;L^p(\Omega)), \label{propreg4}\\
&\varphi_t\in L^2(0,T;V'),\quad\mbox{if}\quad d=2\quad\mbox{ or }\quad\big(d=3 \mbox{ and } p\geq 3\big),\label{propreg5}\\
& u_t\in L^2(0,T;V_{div}'),\quad\mbox{if}\quad d=2.\label{propreg6}
\end{align}
Finally, assume that $d=2$ and (H4) is replaced by (H7). Then,
\begin{enumerate}
\item any weak solution satisfies the energy identity
\begin{equation}
\frac{d}{dt}\mathcal{E}(u,\varphi)
+\nu\|\nabla u \|^2+\|\sqrt{m(\varphi)}\nabla\mu \|^2=\langle h(t),u\rangle,\qquad t>0.
\label{eniden}
\end{equation}
In particular we have $u\in C([0,\infty);G_{div})$, $\varphi\in C([0,\infty);H)$ and $\int_\Omega F(\varphi)\in C([0,\infty))$.
\item If in addition $h\in L^2_{tb}(0,\infty;V_{div}')$,
%and
%$$
%\|h\|_{L^2_{tb}(0,\infty;V_{div}')}:=\Big(\sup_{t\geq 0}
%\int_t^{t+1}\|h(\tau)\|_{V_{div}'}^2 d\tau\Big)^{1/2}
%<\infty,
%$$
then any weak solution satisfies
also the dissipative estimate
\begin{align}
&\mathcal{E}(u(t),\varphi(t))\leq \mathcal{E}(u_0,\varphi_0)e^{-kt}+ F(m_0)|\Omega| + K,
\qquad\forall t\geq 0,\label{disest}
\end{align}
where $m_0=(\varphi_0,1)$ and $k$, $K$ are two positive constants
which are independent of the initial data, with $K$ depending on
$\Omega$, $\nu$, $J$, $F$ and
$\|h\|_{L^2_{tb}(0,\infty;V_{div}')}$.
\end{enumerate}

\end{thm}

%\begin{oss}{\upshape Instead of Assumption (H4) we can consider the following one
%be replaced with the following one

%\begin{description}
%\item[(H4)'] There exist
   % $c_1>\frac{1}{2}\|J\|_{L^1(\mathbb{R}^d)}$ and
   % $c_2\in\mathbb{R}$ such that
            %$$F(s)\geq c_1s^2-c_2,\qquad\forall s\in\mathbb{R}.$$
%\end{description}
%Existence of a weak solution can then still be established. Nevertheless,
%the regularity properties for such solution are slightly different (see \cite[Theorem 1]{CFG}).
%}
%\end{oss}

\begin{proof}
The argument follows the lines of \cite[Proof of Theorem 1]{CFG} and is based
on a Faedo-Galerkin approximation scheme. For the reader's convenience, we give a sketch of it.
Let us first assume that $\varphi_0\in D(B)$, where the operator $B$ is given by
$B=-\Delta+I$ with homogeneous Neumann boundary condition.
We introduce the family $\{w_j\}_{j\geq 1}$ of the eigenfunctions of the Stokes operator $S$ as a Galerkin
base in $V_{div}$ and the family $\{\psi_j\}_{j\geq 1}$ of the eigenfunctions of $B$
as a Galerkin base in $V$. We define the $n-$dimensional subspaces
$\mathcal{W}_n:=\langle w_1,\cdots,w_n\rangle$ and
$\Psi_n:=\langle\psi_1,\cdots,\psi_n\rangle$ and consider the
orthogonal projectors on these subspaces in $G_{div}$ and $H$,
respectively, i.e., $\widetilde{P}_n:=P_{\mathcal{W}_n}$ and
$P_n:=P_{\Psi_n}$. We then look for three functions of the form
$$u_n(t)=\sum_{k=1}^n \alpha^{(n)}_k(t)w_k,\quad \varphi_n(t)=\sum_{k=1}^n \beta^{(n)}_k(t)\psi_k,
\quad \mu_n(t)=\sum_{k=1}^n \gamma^{(n)}_k(t)\psi_k$$ which solve
the following approximating problem
\begin{align}
&(\varphi_n',\psi)+(m(\varphi_n)\nabla\mu_n,\nabla\psi)=(u_n\varphi_n,\nabla\psi)\label{apb1}\\
&(u_n',v)+\nu(\nabla u_n,\nabla v)+b(u_n,u_n,v)=-(\varphi_n\nabla\mu_n,v)+(h_n,v)\label{apb2}\\
&\mu_n=P_n\big(a\varphi_n-J\ast\varphi_n+F'(\varphi_n)\big)\label{apb3}\\
&\varphi_n(0)=\varphi_{0n},\quad u_n(0)=u_{0n},\label{apb4}
\end{align}
for every $\psi\in\Psi_n$ and every $v\in\mathcal{W}_n$, where
$\varphi_{0n}=P_n\varphi_0$ and $u_{0n}=\widetilde{P}_nu_0$
(primes denote derivatives with respect to time). In \eqref{apb2}
$\{h_n\}$ is a sequence in $C^0([0,T];G_{div})$ such that $h_n\to
h$ in $L^2(0,T;V_{div}')$. It is easy to see that this
approximating problem is equivalent to solve a Cauchy problem
for a system of ODEs in the
$2n$ unknowns $\alpha^{(n)}_i$, $\beta^{(n)}_i$
($\gamma^{(n)}_i$ can be deduced from \eqref{apb3}).
  Since $F^{\prime\prime}$ and $m$ are
  locally Lipschitz on $\mathbb{R}$, the
Cauchy-Lipschitz theorem ensures that there exists
$T^{\ast}_n\in(0,+\infty]$ such that system \eqref{apb1}-\eqref{apb4} has a unique
maximal solution
$\textbf{a}^{(n)}:=(\alpha^{(n)}_1,\cdots,\alpha^{(n)}_n)$,
$\textbf{b}^{(n)}:=(\beta^{(n)}_1,\cdots,\beta^{(n)}_n)$ on
$[0,T^{\ast}_n)$ with $\textbf{a}^{(n)}$, $\textbf{b}^{(n)}\in
C^1([0,T^{\ast}_n);\mathbb{R}^n)$.

By taking $\psi=\mu_n$ and $v=u_n$ in \eqref{apb1} and \eqref{apb2}, respectively,
and adding the resulting identities together, we get
\begin{align}
&\frac{d}{dt}\mathcal{E}(u_n,\varphi_n)+\nu\Vert\nabla u_n\Vert^2+
\Vert \sqrt{m(\varphi_n)}\nabla\mu_n\Vert^2=(h_n,u_n),\nonumber
\end{align}
where $\mathcal{E}$ is defined as in \eqref{endef}. Integrating
this identity between $0$ and $t$
\begin{align}
&\mathcal{E}(u_n(t),\varphi_n(t)) +\int_0^t\Big(\nu\|\nabla u_n
\|^2+\|\sqrt{m(\varphi_n)}\nabla\mu_n \|^2\Big)d\tau
=\mathcal{E}(u_{0n},\varphi_{0n})+\int_0^t\langle h_n(\tau),u_n
\rangle d\tau.\label{eiapp}
\end{align}
Observe now that
\begin{align}
&\int_0^t\langle h_n(\tau),u_n \rangle
d\tau\leq\frac{\nu}{2}\int_0^t\Vert\nabla u_n\Vert^2
d\tau+\frac{1}{2\nu}\int_0^t\Vert h_n\Vert_{V_{div}'}^2
d\tau.\nonumber
\end{align}
On the other hand, taking (H2) and (H4) into account,  we get
\begin{align}
&\mathcal{E}(u_n,\varphi_n)=\frac{1}{2}\Vert u_n\Vert^2+
\frac{1}{2}\Vert\sqrt{a}\varphi_n\Vert^2-\frac{1}{2}(\varphi_n,J\ast\varphi_n)
+\int_\Omega F(\varphi_n)\nonumber\\
&\geq\frac{1}{2}\Vert u_n\Vert^2+\frac{1}{2}\int_\Omega a\varphi_n^2
-\frac{1}{2}\Vert\varphi_n\Vert\Vert J\ast\varphi_n\Vert+\int_\Omega(c_1\varphi_n^2-c_2)\nonumber\\
&\geq \frac{1}{2}\Vert u_n\Vert^2+c_1'\Vert\varphi_n\Vert^2-c_2',\nonumber
\end{align}
where $c_1'=(a_\ast-a^\ast)/2+c_1>0$ and $c_2'=c_2|\Omega|$. Using
also \color{black} the convergence assumption for $\{u_{0n}\}$,
$\{h_n\}$, the fact that $\varphi_{0n}\to\varphi_0$ in
$H^2(\Omega)$ (since $\varphi_0\in D(B)$) and the lower bound
$m_1>0$ for the mobility $m$ (cf. (H1)), we first deduce that
$T_n^\ast=+\infty$ for every $n\geq 1$ (notice that
$|\textbf{a}^{(n)}(t)|=\Vert u_n(t)\Vert$ and
$|\textbf{b}^{(n)}(t)|=\Vert \varphi_n(t)\Vert$) and furthermore
we get the following estimates which hold for any given
$0<T<+\infty$
\begin{align}
&\|u_n\|_{L^{\infty}(0,T;G_{div})\cap L^2(0,T;V_{div})}\leq C,\label{est1}\\
&\|\varphi_n\|_{L^{\infty}(0,T;H)}\leq C,\label{est2}\\
&\| F(\varphi_n)\|_{L^{\infty}(0,T;L^1(\Omega))}\leq C,\label{est3}\\
&\|\nabla\mu_n\|_{L^2(0,T;H)}\leq C,\label{est4}
\end{align}
with $C$ independent of $n$. Henceforth we shall denote by $C$ a
positive constant which depends at most on $\Vert u_0\Vert$, $\Vert\varphi_0\Vert$,
$\int_\Omega F(\varphi_0)$, $\Vert
h\Vert_{L^2(0,T;V_{div}')}$ and on $J$, $F$, $\nu$, $m_1$, $\Omega$, but they are independent of $n$. Instead, $c$ will stand for a generic positive constant depending on the parameters of the problem only, i.e. on $J$, $F$, $\nu$, $m_1$ and $\Omega$, but it is independent of $n$.
The values of both $C$ and $c$ may possibly vary even within the same line.

Let us now recall the estimate
\begin{equation}
\label{chemlowbd}
\|\nabla\mu_n\|^2\geq\frac{c_0^2}{4}\|\nabla\varphi_n\|^2-c\|\varphi_n\|^2,
\end{equation}
which can be deduced as in \cite[Proof of Theorem 1]{CFG}
by multiplying $\mu_n$ by $-\Delta\varphi_n$ in $L^2$ and integrating by parts.
By means of \eqref{chemlowbd}, from \eqref{est2} and \eqref{est4} we get
\begin{align}
\|\varphi_n\|_{L^2(0,T;V)}\leq C.\label{est6}
\end{align}
Due to (H5), which in particular implies that $|F'(s)|\leq c|F(s)|+c$, we have
$$|\overline{\mu}_n|=\Big|\int_\Omega F'(\varphi_n)\Big|\leq c\Vert F(\varphi_n)\Vert_{L^1(\Omega)}+c\leq C,$$
%and, since $$
%(due to (H5) which implies that $|F'(s)|\leq c|F(s)|+c$),
and therefore, thanks to \eqref{est3}, \eqref{est4} and to the Poincar\'{e}-Wirtinger inequality, we deduce
\begin{align}
&\|\mu_n\|_{L^2(0,T;V)}\leq C.\label{est7}
\end{align}

In addition, (H5) and \eqref{est3} yield the following control
\begin{align}
&\Vert F'(\varphi_n)\Vert_{L^\infty(0,T;L^r(\Omega))}\leq C.
\label{est8}
\end{align}

Let us now derive the estimates for the sequences of time derivatives $\{u_n'\}$ and $\{\varphi_n'\}$.
By taking $\psi_I:=P_n\psi$, for $\psi\in V$ arbitrary, as test function in \eqref{apb1} we obtain
\begin{align}
&\langle\varphi_n',\psi\rangle=(\varphi_n',\psi_I)=-(m(\varphi_n)\nabla\mu_n,\nabla\psi_I)+(u_n\varphi_n,\nabla\psi_I).
\nonumber
\end{align}

Assume first $d=3$. We have
\begin{align}
&|\langle\varphi_n',\psi\rangle|\leq m_2\Vert\nabla\mu_n\Vert\Vert\nabla\psi_I\Vert+
\Vert u_n\Vert_{L^6(\Omega)^3}\Vert\varphi_n\Vert_{L^3(\Omega)}\Vert\nabla\psi_I\Vert\nonumber\\
&\leq c\big(\Vert\nabla\mu_n\Vert+\Vert\nabla u_n\Vert\Vert\varphi_n\Vert^{1/2}\Vert\varphi_n\Vert_{L^6(\Omega)}^{1/2}\big)\Vert\nabla\psi\Vert\nonumber\\
&\leq C\big(\Vert\nabla\mu_n\Vert+
\Vert\nabla u_n\Vert\Vert\varphi_n\Vert_V^{1/2}\big)\Vert\nabla\psi\Vert,\qquad\forall\psi\in V\label{est9}
\end{align}
where in the last estimate we have used \eqref{est2}.
Then, by \eqref{est1}, \eqref{est4} and \eqref{est6} from \eqref{est9} we get
\begin{align}
&\Vert\varphi_n'\Vert_{L^{4/3}(0,T;V')}\leq C.
\label{est31}
\end{align}
For $d=2$, by arguing as above it is not difficult to infer the following bound
\begin{align}
&\Vert\varphi_n'\Vert_{L^{2-\delta}(0,T;V')}\leq C.\label{est10}
\end{align}

As far as the sequence $\{u_n'\}$ is concerned, by arguing exactly as in \cite[Proof of Theorem 1]{CFG}
we can deduce the bounds
\begin{align}
&\Vert u_n'\Vert_{L^{4/3}(0,T;V_{div}')}\leq C,\qquad\mbox{if}\:\: d=3,\label{est11}\\
& \Vert u_n'\Vert_{L^{2-\gamma}(0,T;V_{div}')}\leq C,\qquad\mbox{if}\:\: d=2.\label{est12}
\end{align}
From \eqref{est1}--\eqref{est2}, \eqref{est6}--\eqref{est8} and \eqref{est31}--\eqref{est12},
using compactness results, we obtain for a not relabeled subsequence
\begin{align}
& u_n\rightharpoonup u\quad\mbox{weakly}^{\ast}\mbox{ in } L^{\infty}(0,T;G_{div}),
\quad\mbox{weakly in }L^2(0,T;V_{div}),\label{con1}\\
& u_n\to u\quad\mbox{strongly in }L^2(0,T;G_{div}),\quad\mbox{a.e. in }Q_T,\label{con2}\\
& u_n'\rightharpoonup u_t\quad\mbox{weakly in }L^{4/3}(0,T;V_{div}'),\qquad d=3,\label{con3}\\
& u_n'\rightharpoonup u_t
\quad\mbox{weakly in }L^{2-\gamma}(0,T;V_{div}'),\qquad\:\: d=2,\label{con4}\\
& \varphi_n\rightharpoonup\varphi\quad\mbox{weakly}^{\ast}\mbox{ in }L^{\infty}(0,T;H),
\quad\mbox{weakly in }L^2(0,T;V),\label{con5}\\
& \varphi_n\to\varphi\quad\mbox{strongly in }L^2(0,T;H),\quad\mbox{a.e. in }Q_T,\label{con6}\\
& \varphi_n'\to\varphi_t\quad\mbox{weakly in }L^{2-\delta}(0,T;V'), \label{con7}\\
& F'(\varphi_n)\rightharpoonup F^\ast\quad\mbox{weakly}^{\ast}\mbox{ in } L^{\infty}(0,T;L^r(\Omega)),\label{con8}\\
& \mu_n\rightharpoonup\mu\quad\mbox{weakly in }L^2(0,T;V).\label{con9}
\end{align}
where $Q_T:=\Omega\times (0,T)$. The pointwise
convergence \eqref{con6}, the weak $^\ast$ convergence
\eqref{con8} and the continuity of $F'$ yield
$F^\ast=F'(\varphi)$.
 By means of
\eqref{con1}--\eqref{con9} we can now pass to the limit in the
approximate problem \eqref{apb1} --\eqref{apb4} as done in
\cite[Proof of Theorem 1]{CFG} and obtain the variational
formulation of system \eqref{sy1}--\eqref{sy6} for functions $u$,
$\varphi$ and $\mu$. In particular notice that, fixing $n$
arbitrarily, for every $\chi\in C^\infty_0(0,T)$ and every
$\psi\in \Psi_n$ we have
\begin{align}
&\int_0^T(m(\varphi_k)\nabla\mu_k,\nabla\psi)\chi d\tau\to\int_0^T(m(\varphi)\nabla\mu,\nabla\psi)\chi d\tau,
\qquad\mbox{as }\:\: k\to\infty,\nonumber
\end{align}
as a consequence of the weak convergence \eqref{con9} and of the convergence
\begin{align}
&m(\varphi_k)\to m(\varphi)\qquad\mbox{strongly in }\: L^r(Q_T),\quad\mbox{for all }\:\: r\in [1,\infty),
\label{con10}
\end{align}
ensured by \eqref{con6} and by Lebesgue's theorem.

As far as the energy inequality is concerned, let us observe that \eqref{con9} and \eqref{con10}
imply that
\begin{align}
&\sqrt{m(\varphi_n)}\nabla\mu_n\rightharpoonup\sqrt{m(\varphi)}\nabla\mu\qquad\mbox{weakly in}\: L^s(Q_T),
\quad\mbox{for all }\: s\in[1,2).
\label{con11}
\end{align}
But \eqref{eiapp} yields
\begin{align}
&\Vert\sqrt{m(\varphi_n)}\nabla\mu_n\Vert_{L^2(Q_T)}\leq C,\nonumber
\end{align}
and hence \eqref{con11} holds also for $s=2$. Therefore
\begin{align}
&\int_0^t\Vert\sqrt{m(\varphi)}\nabla\mu\Vert^2 d\tau
\leq\liminf_{n\to\infty}\int_0^t\Vert\sqrt{m(\varphi_n)}\nabla\mu_n\Vert^2 d\tau.\nonumber
\end{align}
This fact and the above convergences allow us to pass to the inferior limit
in \eqref{eiapp} and deduce \eqref{ei}.

Let us now suppose that assumption (H4) is replaced by (H7). Then, the energy of the approximate solution
can be controlled from below by
\begin{align}
&\mathcal{E}(u_n,\varphi_n)\geq c\big(\Vert u_n\Vert^2+\Vert\varphi_n\Vert_{L^p(\Omega)}^p\big)-c.\nonumber
\end{align}
From \eqref{eiapp} we can then improve \eqref{est2}, that is
\begin{align}
&\Vert\varphi_n\Vert_{L^\infty(0,T;L^p(\Omega))}\leq C,
\label{est32}
\end{align}
which yields \eqref{propreg4}. The regularity properties \eqref{propreg5} and \eqref{propreg6}
follow by a comparison argument exactly as in \cite[Proof of Corollary 1]{CFG} with
reference to the weak formulation \eqref{wdefnd1}, \eqref{wdefnd2} and taking
the improved regularity for $\varphi$ and (H1) into account.

Furthermore, if $d=2$ and (H4) is replaced by (H7), we can take $\mu$ and $u$ as test functions in
the variational formulation \eqref{wdefnd1}, \eqref{wdefnd2} and argue as in \cite[Proof of Corollary 2]{CFG}
in order to deduce the energy identity \eqref{eniden}.
As far as the the dissipative estimate \eqref{disest} is concerned, this can be deduced without difficulties
by adapting the argument of \cite[Proof of Corollary 2]{CFG} and taking (H1) into account.

Finally, if $\varphi_0\in H$ with $F(\varphi_0)\in L^1(\Omega)$, we can
approximate $\varphi_0$ with $\varphi_{0m}\in D(B)$ given by
$\varphi_{0m}:=(I+B/m)^{-1}\varphi_0$. This sequence satisfies
$\varphi_{0m}\to\varphi$ in $H$ and by exploiting assumption (H3) and the argument
of \cite[Proof of Theorem 1]{CFG} we can easily recover the existence of a weak solution, the energy
inequality \eqref{ei} and, for $d=2$, the energy identity \eqref{eniden} as well as the dissipative estimate \eqref{disest}.
\end{proof}

It may be interesting to observe that another energy identity for $d=2$ (inequality for $d=3$)
is satisfied by the weak solution of Theorem \ref{thm}. This energy identity will turn to be useful
especially in the degenerate case in order to establish the existence of the global attractor in 2D.
\begin{cor}\label{cor1}
Let the assumptions of Theorem \ref{thm} be satisfied with (H4) replaced by (H7).
Then, if $d=2$ the weak solution $z=[u,\varphi]$ constructed in Theorem \ref{thm}
satisfies the following energy identity
\begin{align}
&
\frac{1}{2}\frac{d}{dt}
\big(\Vert u\Vert^2+\Vert\varphi\Vert^2\big)+\int_\Omega m(\varphi)F''(\varphi)|\nabla\varphi|^2
+\int_\Omega a m(\varphi)|\nabla\varphi|^2
+\nu\Vert\nabla u\Vert^2\nonumber\\
&=\int_\Omega m(\varphi)(\nabla J\ast\varphi-\varphi\nabla
a)\cdot\nabla\varphi +\int_\Omega(a\varphi-J\ast\varphi)
u\cdot\nabla\varphi+\langle h,u\rangle, \label{energeq}
\end{align}
for almost any $t>0$.
Furthermore, if $d=3$ and if (H7) is satisfied with
$p\geq 3$, the weak solution $z$ satisfies the following energy
inequality
\begin{align}
&\frac{1}{2}\big(\Vert u(t)\Vert^2+\Vert\varphi(t)\Vert^2\big)
+\int_0^t\int_\Omega m(\varphi) F''(\varphi)|\nabla\varphi|^2+\int_0^t\int_\Omega a m(\varphi)|\nabla\varphi|^2
\nonumber\\
&+\nu\int_0^t\Vert\nabla u\Vert^2\leq\frac{1}{2}\big(\Vert u_0\Vert^2+\Vert\varphi_0\Vert^2\big)
+\int_0^t\int_\Omega m(\varphi)\big(\nabla J\ast\varphi-\varphi\nabla a\big)\cdot\nabla\varphi
\nonumber\\
&+\int_0^t\int_\Omega\big(a\varphi-J\ast\varphi\big) u\cdot\nabla\varphi+\int_0^t\langle h,u\rangle d\tau,
\qquad\forall t>0.
\label{energineq}
\end{align}
\end{cor}

\begin{proof}
If $d=2$, we can take $\psi=\varphi$ and $v=u$ as test functions in the weak formulation
\eqref{wdefnd1} and \eqref{wdefnd2}, respectively. This is allowed due to the regularity properties
\eqref{propreg5} and \eqref{propreg6}.
Then, we add the resulting identities and we observe that, since by
an Alikakos' iteration technique as in \cite[Theorem 2.1]{BH1} it can be shown that
$\varphi\in L^\infty(\Omega)$ and since $F\in C^2(\mathbb{R})$,
we can write $\nabla F'(\varphi)=F''(\varphi)\nabla\varphi$. Then,
on account of this identity and of \eqref{propreg7}, the second term on the left hand side of \eqref{wdefnd1} can be rewritten and \eqref{energeq} immediately follows.

Let $d=3$ and assume that (H7) holds with $p\geq 3$. Then, due to the second of \eqref{propreg5}
we can take $\psi=\varphi$ as test function in \eqref{wdefnd1} and get
\begin{align}
&\frac{1}{2}\frac{d}{dt}\Vert\varphi\Vert^2+\big(m(\varphi)\nabla\mu,\nabla\varphi\big)=0.\nonumber
\end{align}
By integrating this identity between 0 and $t$ and by rewriting the term $(m(\varphi)\nabla\mu,\nabla\varphi)$
as done above in the case $d=2$ we obtain
\begin{align}
&\frac{1}{2}\Vert\varphi(t)\Vert^2
+\int_0^t\int_\Omega m(\varphi) F''(\varphi)|\nabla\varphi|^2+\int_0^t\int_\Omega a m(\varphi)|\nabla\varphi|^2
\nonumber\\
&+\int_0^t\int_\Omega m(\varphi)\big(\varphi\nabla a-\nabla J\ast\varphi\big)\cdot\nabla\varphi=\frac{1}{2}\Vert\varphi_0\Vert^2.
\label{halfid1}
\end{align}

On the other hand we take $v=u_n$ in the second equation \eqref{apb2}
of the Faedo-Galerkin approximating problem and integrate the resulting identity to get
\begin{align}
&\frac{1}{2}\Vert u_n(t)\Vert^2+\nu\int_0^t\Vert\nabla u_n\Vert^2 d\tau=
\frac{1}{2}\Vert u_{0n}\Vert^2-\int_0^t\big(\varphi_n\nabla\mu_n,u_n\big)d\tau
+\int_0^t\langle h,u_n\rangle d\tau\nonumber\\
&=\frac{1}{2}\Vert u_{0n}\Vert^2+\int_0^t\int_\Omega\big(a\varphi_n-J\ast\varphi_n\big)\nabla\varphi_n\cdot u_n
+\int_0^t\langle h,u_n\rangle d\tau.
\label{halfid}
\end{align}
We can now pass to the limit in \eqref{halfid}
and use the second weak convergence \eqref{con1} and \eqref{con2}. In particular, observe that
\begin{align}
&\int_0^t\int_\Omega\big(a\varphi_n-J\ast\varphi_n\big)\nabla\varphi_n\cdot u_n
=-\int_{Q_t}\Big(\frac{1}{2}\varphi_n^2\nabla a-\varphi_n\big(\nabla J\ast\varphi_n\big)\Big)\cdot u_n\nonumber\\
&\to-\int_{Q_t}\Big(\frac{1}{2}\varphi^2\nabla a-\varphi\big(\nabla J\ast\varphi\big)\Big)\cdot u
=\int_0^t\int_\Omega\big(a\varphi-J\ast\varphi\big)\nabla\varphi\cdot u,\nonumber
\end{align}
where this last convergence is a consequence of the fact that, due to \eqref{con1}, \eqref{con2},
\eqref{con5}, \eqref{con6}
and interpolation, we have $\varphi_n\to\varphi$ strongly in $L^3(Q_t)$ and $u_n\to u$
strongly in $L^3(Q_t)^3$ (recall that $\nabla a\in L^\infty$). Hence we get
\begin{align}
&\frac{1}{2}\Vert u(t)\Vert^2
+\nu\int_0^t\Vert\nabla u\Vert^2\leq\frac{1}{2}\Vert u_0\Vert^2
+\int_0^t\int_\Omega\big(a\varphi-J\ast\varphi\big) u\cdot\nabla\varphi+\int_0^t\langle h,u\rangle d\tau.
\label{halfid2}
\end{align}
Summing \eqref{halfid1} and \eqref{halfid2} we deduce \eqref{energineq}.
\end{proof}

\begin{oss}
{\upshape The theorems proven in this section still hold when the viscosity smoothly depends on $\varphi$
(see \cite{FG1}). Moreover, they allow us to generalize to the variable
(non-degenerate) mobility the results obtained in \cite{FG1} on the longtime
behavior. More precisely, the existence of the global attractor
and the existence of a trajectory attractor in the cases $d=2$ and $d=3$, respectively.}
\end{oss}

\section{Degenerate mobility}\setcounter{equation}{0}
\label{sec:deg}

In this section we consider a mobility $m$ which degenerates at $\pm1$
and we assume that the double-well potential $F$ is singular (e.g.
logarithmic like) and defined in $(-1,1)$. More precisely, we
assume that $m\in C^1([-1,1])$, $m\geq 0$ and that $m(s)=0$ if and
only if $s=-1$ or $s=1$. Furthermore, we suppose
that $m$ and $F$ fulfill the condition
\begin{description}

\item [(A1)] $F\in C^2(-1,1)$ and
\begin{align}
& mF''\in C([-1,1]).\nonumber
\end{align}
%There exists $\alpha>0$ such that
%\begin{align}
%& m(s)F''(s)\geq \alpha,\qquad\forall s\in[-1,1].
%\end{align}

\end{description}
We point out that (A1) is a typical condition which
arises in the Cahn-Hilliard equation with degenerate mobility (see \cite{EG,GL1,GL2,GZ}).

As far as $F$ is concerned
%Furthermore, we assume that the singular potential $F$ can be
we assume that it can be written in the following form
$$F=F_1+F_2,$$
%where $F_2\in C^2([-1,1])$ and $F_1$ is such that there exists a positive integer $q$
%such that $F_1\in C^{(2+2q)}(-1,1)$ and the following assumptions are satisfied
%(see \cite{FG2})
where the singular component $F_1$ and
the regular component $F_2\in
C^2([-1,1])$ satisfy the following assumptions.

\begin{description}
\item[(A2)]
There exist $a_2>4(a^\ast-a_\ast-b_2)$, where $b_2:=\min_{[-1,1]}F_2''$,
%There exists $a_2>4(a^\ast-a_\ast)$
 and $\e_0>0$ such that
\begin{align*}
&F_1^{''}(s)\geq a_2,\qquad\forall s\in(-1,-1+\e_0]\cup[1-\e_0,1).
%\mbox{near }s=\pm 1.\\
\end{align*}

%\item[(A3)] There exists $\e_0>0$ such that, for each
%    $k=0,1,\cdots, 2+2q$ and each $j=0,1,\cdots, q$,
%\begin{align*}
%&F_1^{(k)}(s)\geq 0,\qquad\forall s\in[1-\e_0,1),\\
%\mbox{near }s=1,\\
%&F_1^{(2j+2)}(s)\geq 0,\qquad F_1^{(2j+1)}(s)\leq 0,\qquad\forall s\in(-1,-1+\e_0].
%\mbox{near }s=-1.
%\end{align*}

\item[(A3)] There exists $\e_0>0$ such that $F_1^{''}$ is
    non-decreasing in $[1-\e_0,1)$ and non-increasing in $(-1,-1+\e_0]$.
\item[(A4)]There exists
$c_0>0$ such that
\begin{align}
&F''(s)+a(x)\geq c_0,\qquad\forall s\in(-1,1),\qquad\mbox{a.a. }x\in\Omega.
\end{align}
\end{description}

%\item[(A4)]There exist $\beta,\gamma\in\mathbb{R}$ with $\beta+\gamma>-\min_{[-1,1]}F_2''$ such that
%\begin{equation*}
%F_1^{''}(s)\geq\beta,\qquad\forall s\in(-1,1),\quad
%a(x)\geq \gamma,\qquad\mbox{a.e. }x\in\Omega.
%\end{equation*}

The constants $a^\ast$ and $a_\ast$ are given in (H2) and (H4), respectively, and
 the assumption on the external force $h$ is still (H6).
 %\color{black} Note that (A4)
 %is more general than \cite[(A6)]{FG2}. \color{black}

 \begin{oss}{\upshape
 It is easy to see that (A1)--(A4) are satisfied in the physically
 relevant case where the mobility and the double-well potential
 are given by
 \begin{align}
& m(s)=k_1(1-s^2),\qquad
F(s)=-\frac{\theta_c}{2}s^2+\frac{\theta}{2}\big((1+s)\log(1+s)+(1-s)\log(1-s)\big),
\label{potlog}
 \end{align}
where $0<\theta<\theta_c$. Indeed, setting
$F_1(s):=(\theta/2)\big((1+s)\log(1+s)+(1-s)\log(1-s)\big)$ and
$F_2(s)=-(\theta_c/2)s^2$, then we have $mF_1''=k_1\theta>0$ and so
(A1) is fulfilled. Moreover $F_1$ satisfies also (A2) and (A3),
while (A4) holds if and only if $\inf_\Omega a>\theta_c-\theta$.
Another example is given by taking
 \begin{align}
& m(s)=k(s)(1-s^2)^m,\qquad F(s)=-k_2 s^2+F_1(s) \nonumber
 \end{align}
where $k\in C^1([-1,1])$ such that $0<k_3\leq k(s)\leq k_4$ for
all $s\in[-1,1]$, and $F_1$ is a $C^2(-1,1)$ convex function such
that
\begin{align}
&F_1''(s)=l(s)(1-s^2)^{-m},\qquad\forall s\in(-1,1),\nonumber
\end{align}
where $m\geq 1$ and $l\in C^1([-1,1])$.}
 \end{oss}

 \begin{oss}
 {\upshape
Note that (A4),
which is equivalent to the condition $\inf_\Omega a>-\inf_{(-1,1)}F''$,
 is more general than \cite[(A6)]{FG2}.
 More precisely, \cite[(A6)]{FG2} implies (A4), while (A4)
 implies \cite[(A6)]{FG2} provided we have $\inf_{(-1,1)}(F_1''+F_2'')
 =\inf_{(-1,1)}F_1''+\min_{[-1,1]}F_2''$. For example, consider
 the following double-well potential
 \begin{align}
 &F(s)=-\frac{\theta_c}{2}s^2-\frac{\theta_2}{12}s^4+\frac{\theta}{2}\big((1+s)\log(1+s)+(1-s)\log(1-s)\big),
 \nonumber
 \end{align}
 where $0<\theta<\theta_c$ and $0<\theta<\theta_2$.
 Then, it easy to see that
 \cite[(A6)]{FG2} is satisfied iff $\inf_\Omega a>\theta_c+\theta_2-\theta$,
 while (A4) requires the weaker condition
 $\inf_\Omega a>\theta_c+\theta_2-2\sqrt{\theta\theta_2}$.
}
\end{oss}

%In the degenerate case
If the mobility degenerates we are no longer able to control
the gradient of the chemical potential $\mu$ in some $L^p$
space. For this reason, and also in order to pass to the limit
in the approximate problem considered in the proof of the existence result,
we shall have to suitably reformulate
the definition of weak solution in such a way that $\mu$
does not appear anymore (cf. \cite{EG}).

\begin{defn}
\label{wsdeg}
Let $u_0\in G_{div}$, $\varphi_0\in H$ with $F(\varphi_0)\in L^1(\Omega)$ and $0<T<+\infty$ be given.
A couple $[u,\varphi]$ is a weak solution to \eqref{sy1}-\eqref{sy6} on $[0,T]$ corresponding to $[u_0,\varphi_0]$
if
\begin{itemize}
\item  $u$, $\varphi$ satisfy
\begin{align}
&u\in L^{\infty}(0,T;G_{div})\cap L^2(0,T;V_{div}),\label{reg1}\\
&u_t\in L^{4/3}(0,T;V_{div}'),\qquad\mbox{if}\quad d=3,\label{reg2}\\
&u_t\in L^2(0,T;V_{div}'),\qquad\mbox{if}\quad d=2,\label{reg3}\\
&\varphi\in L^{\infty}(0,T;H)\cap L^2(0,T;V),\label{reg4}\\
&\varphi_t\in L^2(0,T;V'),\label{reg5}
\end{align}
and
\begin{align}
&\varphi\in L^{\infty}(Q_T),\qquad|\varphi(x,t)|\leq 1\quad\mbox{a.e. }(x,t)\in Q_T:=\Omega\times(0,T);\label{reg6}
\end{align}

\item for every $\psi\in V$, every $v\in V_{div}$ and for almost
any $t\in(0,T)$ we have
\begin{align}
&\langle\varphi_t,\psi\rangle+\int_\Omega m(\varphi)F''(\varphi)\nabla\varphi\cdot\nabla\psi+
\int_\Omega m(\varphi) a \nabla\varphi\cdot\nabla\psi\nonumber\\
&+\int_\Omega m(\varphi)(\varphi\nabla a-\nabla J\ast\varphi)\cdot\nabla\psi=(u\varphi,\nabla\psi),
\label{wform1}\\
&\langle u_t,v\rangle+\nu(\nabla u,\nabla v)+b(u,u,v)=\big((a\varphi-J\ast\varphi)\nabla\varphi,v\big)+\langle h,v\rangle;\label{wform2}
\end{align}

\item the initial conditions $u(0)=u_0$, $\varphi(0)=\varphi_0$ hold
 (cf. Remark \ref{inconds}).

\end{itemize}
\end{defn}

We now state the existence result for the degenerate mobility
case. To this aim we need to introduce the {\itshape \color{black} entropy \color{black}} function $M\in
C^2(-1,1)$ defined by
\begin{align}
&m(s)M''(s)=1,\qquad M(0)=M'(0)=0.\nonumber
\end{align}

\begin{thm}
\label{Theor2} Assume that (A1)-(A4) and (H2), (H6) are satisfied.
% for some fixed positive integer $q$.
Let $u_0\in G_{div}$, $\varphi_0\in L^{\infty}(\Omega)$ such that
$F(\varphi_0)\in L^1(\Omega)$ and $M(\varphi_0)\in L^1(\Omega)$.
Then, for every $T>0$ there exists a weak solution
$z:=[u,\varphi]$ to \eqref{sy1}-\eqref{sy6} on $[0,T]$
corresponding to $[u_0,\varphi_0]$  in the sense of Definition \ref{wsdeg}
\color{black}
such that $\overline{\varphi}(t)=\overline{\varphi_0}$ for all $t\in[0,T]$
and
\begin{align}
&\varphi\in L^{\infty}(0,T;L^p(\Omega)),
\label{phiLq}
\end{align}
where $p\leq 6$ for $d=3$ and $2\leq p<\infty$ for $d=2$. In
addition, if $d=2$, the weak solution $z:=[u,\varphi]$ satisfies
the energy equation \eqref{energeq}, and, if
$d=3$, $z$ satisfies the energy inequality \eqref{energineq}.
%\begin{align}
%&
%\frac{1}{2}\frac{d}{dt}
%\big(\Vert u\Vert^2+\Vert\varphi\Vert^2\big)+\int_\Omega m(\varphi)F''(\varphi)|\nabla\varphi|^2
%+\int_\Omega a m(\varphi)|\nabla\varphi|^2
%+\nu\Vert\nabla u\Vert^2\nonumber\\
%&=\int_\Omega m(\varphi)(\nabla J\ast\varphi-\varphi\nabla a)\cdot\nabla\varphi
%+\int_\Omega(a\varphi-J\ast\varphi) u\cdot\nabla\varphi+\langle h,u\rangle.
%\label{energeq}
%\end{align}

\end{thm}

\color{black}
\begin{oss}
{\upshape
The potential $F$ and the entropy $M$ are not independent. Indeed, assumption (A1) implies that there exists
a constant $\gamma_0>0$ such that $|m(s)F''(s)|\leq \gamma_0$, for all $s\in[-1,1]$.
Combining this estimate with the definition of $M$ we get $|F''(s)|\leq\gamma_0 M''(s)$, for all $s\in(-1,1)$.
Thus we get
\begin{align}
&|F(s)|\leq|F(0)|+|F'(0)||s|+\gamma_0 M(s),\qquad\forall s\in (-1,1).\label{yu}
\end{align}
Therefore, in the statement of Theorem \ref{Theor2}, condition $F(\varphi_0)\in L^1(\Omega)$ is actually a consequence
of the entropy assumption $M(\varphi_0)\in L^1(\Omega)$.
}
\end{oss}
\color{black}

\begin{proof}
Let us consider the approximate problem $P_{\epsilon}$: find a weak solution $[\ue,\phie]$ to
\begin{align}
&\phie'+\ue\cdot\nabla\phie=\mbox{div}(m_\e(\phie)\nabla\mue),\label{Pbe1}\\
&\ue'-\nu\Delta\ue+(\ue\cdot\nabla)\ue
+\nabla\pi_{\epsilon}
=\mue\nabla\phie+h,\label{Pbe2}\\
&\mue=a\phie-J\ast\phie+\Fe'(\phie),\label{Pbe3}\\
&\mbox{div}(\ue)=0,\label{Pbe4}\\
&\frac{\partial\mue}{\partial n}=0,\quad\ue=0,
\quad\mbox{on }\partial\Omega,\label{Pbe5}\\
&\ue(0)=u_0,\quad\phie(0)=\varphi_0,\quad\mbox{in }\Omega.\label{Pbe6}
\end{align}
Problem $P_{\epsilon}$ is obtained from \eqref{sy1}-\eqref{sy6}
by replacing the singular potential $F$ with a smooth potential
$\Fe$ and the degenerate mobility $m$ with a non-degenerate one
$\me$. \color{black} In particular, $\Fe$ is represented by
\color{black}
%In particular, $\Fe$ is represented by
$$\Fe=\Fie+F_{2\e},$$
where $\Fie$  and $F_{2\e}$ are defined by
\begin{align}
&\Fie^{''}(s)=\left\{\begin{array}{lll}
F_1^{''}(1-\e),\qquad s\geq 1-\e\\
F_1^{''}(s),\qquad|s|\leq 1-\e\\
F_1^{''}(-1+\e),\qquad s\leq -1+\e
\end{array}\right. \label{approxpot1} \\
& F_{2\e}^{''}(s)=\left\{\begin{array}{lll}
F_2^{''}(1-\e),\qquad s\geq 1-\e\\
F_2^{''}(s),\qquad|s|\leq 1-\e\\
F_2^{''}(-1+\e),\qquad s\leq -1+\e,
\end{array}\right.
 \label{approxpot2}
\end{align}
with $\Fie(0)=F_1(0)$, $\Fie'(0)=F_1'(0)$,
%$\dots$$\Fie^{(1+2q)}(0)=F_1^{(1+2q)}(0)$,
$F_{2\e}(0)=F_2(0)$, $F_{2\e}'(0)=F_2'(0)$.

The approximate non-degenerate mobility is given by
\begin{equation}
\me(s)=\left\{\begin{array}{lll}
m(1-\e),\qquad s\geq 1-\e\\
m(s),\qquad|s|\leq 1-\e\\
m(-1+\e),\qquad s\leq -1+\e.
\end{array}\right.
\label{approxmob}
\end{equation}
Assumption (A3) implies that
\begin{align}
&\Fie(s)\leq F_1(s),\qquad\forall
s\in(-1,1),\qquad\forall\e\in(0,\e_0]. \label{est26}
\end{align}
On the other hand, from the definition of $F_{2\e}$ we have, for
$s>1-\e$,
\begin{align}
&F_{2\e}(s)=F_2(1-\e)+F_2'(1-\e)\big(s-(1-\e)\big)+\frac{1}{2}F_2''(1-\e)\big(s-(1-\e)\big)^2,
\label{F2eps1}
\end{align}
(a similar expression holds for $s<-1+\e$) and $F_{2\e}(s)=F_2(s)$
for $|s|\leq 1-\e$. Since $F_2\in C^2([-1,1])$, then we deduce
that there exist two positive constants $L_1,L_2$, which are
independent of $\e$, such that
\begin{align}
&|F_{2\e}(s)|\leq L_1 s^2+L_2,\qquad\forall s\in\mathbb{R}.
\label{F2eps}
\end{align}
 Therefore, by using the assumption on the initial datum $\varphi_0$
 and \eqref{est26}, \eqref{F2eps}
 we have
\begin{align}
&\int_\Omega\Fe(\varphi_0)\leq\int_\Omega F_1(\varphi_0)
+L_1\Vert\varphi_0\Vert^2+L_2<\infty,\qquad\forall\e\in(0,\e_0].
\label{est22}
\end{align}
  %$\Fe(\varphi_0)\in L^1(\Omega)$ for $\e\in(0,\e_0]$.
Furthermore, we can see that by assumption (A2) there is a $\delta_0\in\mathbb{R}$ such that
\begin{align}
&\Fe(s)\geq\delta s^2-\delta_0,\qquad\forall s\in\mathbb{R},\qquad\forall\e\in(0,\e_0],
\label{est20}
\end{align}
where $0<\delta<a_2/8+b_2/2$.
Indeed, (A2) implies that there exist $a_1^+,a_1^-,a_0\in\mathbb{R}$ such that
$F_1'(s)\geq a_1^+$ for all $s\in[1-\e_0,1)$, $F_1'(s)\leq a_1^-$ for all $s\in(-1,-1+\e_0]$,
and $F_1(s)\geq a_0$ for all $s\in(-1,1)$. Hence, by using definition \eqref{approxpot1},
for $s\geq 1-\e$ we get
\begin{align}
&\Fie(s)=F_1(1-\e)+F_1'(1-\e)\big(s-(1-\e)\big)+\frac{1}{2}F_1''(1-\e)\big(s-(1-\e)\big)^2\nonumber\\
&\geq a_0'+\frac{a_2}{4}\big(s-(1-\e)\big)^2\geq\frac{a_2}{8}s^2+a_0'-\frac{a_2}{2},\label{est19}
\end{align}
provided that $\e\in(0,\e_0]$. For $|s|\leq 1-\e$ we have $\Fie(s)=F_1(s)\geq a_0\geq (a_2/8)s^2-a_2/8+a_0$.
For $s\leq -1+\e$ we can argue as in \eqref{est19}.
On the other hand, from \eqref{F2eps1} we have
\begin{align}
&F_{2\e}(s)\geq \Big(\delta-\frac{a_2}{8}\Big)s^2-\delta_1,\qquad\forall s\in \mathbb{R},\qquad\forall\e\in(0,1),
\label{F2eps2}
\end{align}
where $\delta$ is taken as above and $\delta_1$ is a nonnegative
constant (depending on $b_0:=\min_{[-1,1]}F_2$, $b_1:=\min_{[-1,1]}F_2'$ and $b_2$).
Combining \eqref{est19} with \eqref{F2eps2} we deduce \eqref{est20}.

%Using also \eqref{F_2} we deduce \eqref{est20}.

Moreover, (A4)  immediately  implies that, for $\e\in(0,\e_0]$, there holds
\begin{align}
\Fe^{\prime\prime}(s)+a(x)\geq c_0,\qquad\forall s\in\mathbb{R},\quad\mbox{a.e. }x\in\Omega.
\label{est21}
\end{align}
Indeed, recall that for $s\gtrless\pm 1\mp\e$ we have $\Fie''(s)=F_1''(\pm 1\mp\e)$,
$F_{2\e}''(s)=F_2''(\pm 1\mp\e)$,
and for $|s|\leq 1-\e$ we have $\Fie''(s)=F_1''(s)$, $F_{2\e}''(s)=F_2''(s)$.

We can now check that assumptions (H1)--(H6) of Theorem \ref{thm} for Problem P$_\e$
(with a fixed $\e\in(0,\e_0]$)
are satisfied. In particular,  due to (A2),
we can choose $\delta$ such that
$(a^\ast-a_\ast)/2<\delta<a_2/8+b_2/2$
%we have $a_2/8>(a^\ast-a_\ast)/2$
 and so
(H4) is ensured by \eqref{est20}, while (H5) is satisfied with $r=2$.
Theorem \ref{thm} therefore entails that Problem P$_\e$ admits a weak solution
$z_\e=[\ue,\phie]$ satisfying \eqref{propreg1}--\eqref{propreg3}
and the energy inequality
\begin{equation}
\mathcal{E}_\e\big(\ze(t)\big) +\int_0^t\Big(\nu\|\nabla\ue \|^2+\|\sqrt{\me(\phie)}\nabla\mue \|^2\Big)d\tau
\leq\mathcal{E}_\e(z_0)+\int_0^t\langle h(\tau),\ue(\tau) \rangle d\tau,\label{eiappeps}
\end{equation}
for every $t>0$,
where
\begin{align}
&\mathcal{E}_\e(\ze):=\frac{1}{2}\Vert\ue\Vert^2+\frac{1}{2}\Vert\sqrt{a}\phie\Vert^2
-\frac{1}{2}(\phie,J\ast\phie)+\int_\Omega\Fe(\phie).\nonumber
\end{align}
Now, using \eqref{est22} and \eqref{est20}, from \eqref{eiappeps} we immediately obtain the following
uniform (with respect to $\e\in(0,\e_0]$) estimates
%Now, it can be seen (cf. \cite{FG2}) that assumptions (A2) and (A3) imply that
%there exist some constants $c_q,d_q>0$, which depend on $q$
%but are independent of $\e$, and $\e_0>0$
%such that
%\begin{equation}
%\Fe(s)\geq c_q|s|^{2+2q}-d_q,\qquad\forall s\in\mathbb{R},\quad\forall\e\in(0,\e_0].
%\label{appotlbd}
%\end{equation}
%By using \eqref{appotlbd} the approximate energy inequality \eqref{eiappeps}
%entails the following estimates
\begin{align}
&\|\ue\|_{L^{\infty}(0,T;G_{div})\cap L^2(0,T;V_{div})}\leq C,\label{estep1}\\
&\|\phie\|_{L^{\infty}(0,T;H)}\leq C,\label{estep2}\\
&\| \sqrt{\me(\phie)}\nabla\mue\|_{L^2(0,T;H)}\leq C,\label{estep3}
\end{align}
where the positive constant $C$ can possibly change
from line to line but is always independent of $\e$.
 Now, test
\eqref{Pbe1} by $\psi=M_\e'(\phie)$, where $M_\e$ is a $C^2$
function such that $\me(s)M_\e''(s)=1$ and $M_\e(0)=M_\e'(0)=0$.
We get
\begin{align}
&\frac{d}{dt}\int_\Omega
M_\e(\phie)+\int_\Omega\me(\phie)\nabla\mue\cdot
M_\e''(\phie)\nabla\phie=\int_\Omega\ue\phie\cdot
M_\e''(\phie)\nabla\phie=0, \label{entropep}
\end{align}
%But $\int_\Omega\ue\phie\cdot G_\e''(\phie)\nabla\phie=0$,
where the last identity is due to
\eqref{Pbe4}. Therefore \eqref{entropep} yields
\begin{align}
&\frac{d}{dt}\int_\Omega M_\e(\phie)+\int_\Omega\nabla\mue\cdot
\nabla\phie=0.\nonumber
\end{align}
On account of \eqref{Pbe3}, we obtain
\begin{align}
&\frac{d}{dt}\int_\Omega M_\e(\phie)+
\int_\Omega\Big(\big(a+\Fe''(\phie)\big)|\nabla\phie|^2+\phie\nabla
a\cdot\nabla\phie-\nabla
J\ast\phie\cdot\nabla\phie\Big)=0.\label{entrop2}
\end{align}

%Now, it is not difficult to see (cf. \cite{FG2}) that assumptions (A3) and (A5) imply
%that there exists $c_0:=\beta+\gamma+\min_{[-1,1]}F_2''$ and $\e_0>0$ such that $\Fe''(s)+a(x)\geq c_0$, for all $s\in\mathbb{R}$,
%a.a $x\in\Omega$ and for all $\e\in(0,\e_0]$.
By using \eqref{est21} and \eqref{estep2}, from \eqref{entrop2} we hence infer
\begin{align}
&\frac{d}{dt}\int_\Omega
M_\e(\phie)+\frac{c_0}{2}\Vert\nabla\phie\Vert^2\leq
c_J\Vert\phie\Vert^2\leq C,\nonumber
\end{align}
where $c_J$ is a positive constant depending on $J$ only. Therefore,
on account of the fact that for $\e$ small enough we have
 $M_\e(s)\leq M(s)$ for all $s\in(-1,1)$,
and recalling that $\int_\Omega M(\varphi_0)<\infty$, we deduce
the bounds (see also \eqref{estep2})
\begin{align}
&\Vert\phie\Vert_{L^2(0,T;V)}\leq C,\label{estep4}\\
&\Vert M_\e(\phie)\Vert_{L^\infty(0,T;L^1(\Omega))}\leq
C.\label{estep5}
\end{align}

Let us now establish the estimates for $\phie'$ and $\ue'$.
Let us first consider the case $d=3$
and start from $\phie'$.
From the variational formulation of \eqref{Pbe1} we
get
\begin{align}
&|\langle\phie',\psi\rangle|\leq\Vert\me(\phie)\nabla\mue\Vert\Vert\nabla\psi\Vert
+\Vert\ue\Vert_{L^6(\Omega)}\Vert\phie\Vert_{L^3(\Omega)}\Vert\nabla\psi\Vert\nonumber\\
&\leq c\Vert\sqrt{\me(\phie)}\nabla\mue\Vert\Vert\nabla\psi\Vert+c\Vert\nabla\ue\Vert
\Vert\phie\Vert^{1/2}\Vert\phie\Vert_{L^6(\Omega)}^{1/2}
\Vert\nabla\psi\Vert\nonumber\\
&\leq C\big(\Vert\sqrt{\me(\phie)}\nabla\mue\Vert+\Vert\nabla\ue\Vert\Vert\phie\Vert_V^{1/2}\big)
\Vert\nabla\psi\Vert,\qquad\forall\psi\in V,
\label{est23}
\end{align}
where we have used \eqref{estep2}. Then, on account of \eqref{estep1}, \eqref{estep3} and
\eqref{estep4}, from \eqref{est23} we obtain
%Then, see \eqref{estep1} and the fact that
%$\Vert\me(\phie)\nabla\mue\Vert_{L^2(0,T;H)}\leq C$
%(by \eqref{estep3} and the definition of the approximate mobility \eqref{approxmob})
%we obtain
\begin{align}
&\Vert\phie'\Vert_{L^{4/3}(0,T;V')}\leq C.
\label{estep6}
\end{align}
In order to deduce a bound for $\ue'$, observe first that we have
\begin{align}
&\langle\mue\nabla\phie,v\rangle=\langle(a\phie-J\ast\phie)\nabla\phie,v\rangle,\qquad\forall v\in V_{div}.
\nonumber
\end{align}
Thus, by \eqref{estep2}, we have that
\begin{align}
&|\langle\mue\nabla\phie,v\rangle|\leq\Vert a\phie-J\ast\phie\Vert_{L^3(\Omega)}\Vert\nabla\phie\Vert\Vert v\Vert_{L^6(\Omega)}\nonumber\\
&\leq
%2c_i\Vert J\Vert_{L^1}
c_J\Vert\phie\Vert_{L^3(\Omega)}\Vert\phie\Vert_V
\Vert v\Vert_{V_{div}}\leq c_J\Vert\phie\Vert^{1/2}\Vert\phie\Vert_{L^6(\Omega)}^{1/2}
\Vert\phie\Vert_V\Vert v\Vert_{V_{div}}\nonumber\\
&\leq C\Vert\phie\Vert_V^{3/2}\Vert v\Vert_{V_{div}},
\label{est24}
\end{align}
% ($c_i$ is the injection constant of $V_{div}$ into $L^6$).
Hence, \eqref{estep4} entails the estimate for the Korteweg force in \eqref{Pbe2}
\begin{align}
&\Vert\mue\nabla\phie\Vert_{L^{4/3}(0,T;V_{div}')}\leq C.\nonumber
\end{align}
The estimates for the other terms in the variational formulation of \eqref{Pbe2}
can be deduced by standard arguments for the 3D Navier-Stokes and finally we obtain
\begin{align}
&\Vert \ue'\Vert_{L^{4/3}(0,T;V_{div}')}\leq C.\label{estep7}
\end{align}
In the case $d=2$, by arguing as in \eqref{est23} and \eqref{est24}
and again using \eqref{estep1}--\eqref{estep3} and \eqref{estep4},
it not difficult to see that the following estimates hold
\begin{align}
&\Vert\ue'\Vert_{L^{2-\gamma}(0,T;V_{div}')}\leq C,\qquad
\Vert\phie'\Vert_{L^{2-\delta}(0,T;V')}\leq C,
\qquad\forall\gamma,\delta\in(0,1).\label{estep8}
\end{align}
%\eqref{estep6} still holds and, instead of \eqref{estep7} we get
%\begin{align}
%&\Vert \ue'\Vert_{L^2(0,T;V_{div}')}\leq C.\label{estep8}
%\end{align}
From the estimates above and using standard compactness results we deduce that
there exist $u\in L^{\infty}(0,T;G_{div})\cap L^2(0,T;V_{div})$ and
$\varphi\in L^{\infty}(0,T;H)\cap L^2(0,T;V)$ such that,
up to a (not relabeled) subsequence we have
\begin{align}
& \ue\rightharpoonup u\qquad\mbox{weakly}^{\ast}\mbox{ in } L^{\infty}(0,T;G_{div}),
\quad\mbox{weakly in }L^2(0,T;V_{div}),\label{conve1}\\
& \ue\to u\quad\mbox{strongly in }L^2(0,T;G_{div}),\quad\mbox{a.e. in }Q_T,\label{conve2}\\
& \ue'\rightharpoonup u_t\quad\mbox{weakly in }L^{4/3}(0,T;V_{div}'),\quad\mbox{if }\: d=3,\label{conve3}\\
& \ue'\rightharpoonup u_t
\qquad\mbox{weakly in }L^{2-\gamma}(0,T;V_{div}'),\qquad\gamma\in(0,1),\quad\mbox{if }\: d=2,\label{conve4}\\
& \phie\rightharpoonup\varphi\quad\mbox{weakly}^{\ast}\mbox{ in }L^{\infty}(0,T;H),
\qquad\mbox{weakly in }L^2(0,T;V),\label{conve5}\\
& \phie\to\varphi\quad\mbox{strongly in }L^2(0,T;H),\qquad\mbox{a.e. in }Q_T,\label{conve6}\\
& \phie'\to\varphi_t\qquad\mbox{weakly in }L^{4/3}(0,T;V'),\quad\mbox{if }\: d=3,\label{conve12}\\
& \phie'\to\varphi_t\qquad\mbox{weakly in
}L^{2-\delta}(0,T;V'),\qquad\delta\in(0,1),\quad\mbox{if }\: d=2.
\label{conve7}
\end{align}
Let us now show that $|\varphi|\leq 1$ almost everywhere in $Q_T$. In order to
do that we can argue as in \cite[Proof of Theorem 1]{EG} (see also \cite[Proof of Theorem 2.3]{B}). More precisely, in \cite[Proof of Theorem 1]{EG} the following estimates are established
\begin{align}
& M_\e(s)\geq\frac{1}{2 m(1-\e)}(s-1)^2,\qquad\forall s>1,\nonumber\\
& M_\e(s)\geq\frac{1}{2 m(-1+\e)}(s+1)^2,\qquad\forall
s<-1.\nonumber
\end{align}
Therefore
\begin{align}
&\int_\Omega\big(|\phie|-1\big)_+^2=\int_{\{\phie>1\}}(\phie-1)^2
+\int_{\{\phie<-1\}}(\phie+1)^2\nonumber\\
&\leq 2 m(1-\e)\int_{\{\phie>1\}}M_\e(\phie)+ 2 m(-1+\e)\int_{\{\phie<-1\}}M_\e(\phie)\nonumber\\
&\leq 2\max\big(m(1-\e),m(-1+\e)\big)\int_\Omega
M_\e(\phie).\nonumber
\end{align}
Now, by using \eqref{estep5}, the fact that $m(\pm 1 \mp \e)\to 0$ as $\e\to 0$
and the generalized Lebesgue theorem, we obtain
$$\int_\Omega\big(|\varphi|-1\big)_+^2=0\qquad\mbox{ for a.a. }\:\:t\in(0,T),$$
so that
\begin{align}
&|\varphi(x,t)|\leq 1,\qquad\mbox{ a.e. in }\:\:\Omega,\qquad\mbox{ a.a. }\:\:t\in(0,T).
\label{est25}
\end{align}

We now have to pass to the limit in the variational formulation of the approximate problem
\eqref{Pbe1}--\eqref{Pbe6} in order to show that $[u,\varphi]$ is a weak solution
to \eqref{sy1}--\eqref{sy6} according to Definition \ref{wsdeg}.
It is easy that the variational formulation of \eqref{Pbe1}--\eqref{Pbe6} gives
\begin{align}
&\langle\phie',\psi\rangle+\int_\Omega \me(\phie)\Fe''(\phie)\nabla\phie\cdot\nabla\psi+
\int_\Omega \me(\phie) a \nabla\phie\cdot\nabla\psi\nonumber\\
&+\int_\Omega \me(\phie)(\phie\nabla a-\nabla J\ast\phie)\cdot\nabla\psi=(\ue\phie,\nabla\psi),
\label{wsappb1}\\
&\langle \ue',v\rangle+\nu(\nabla\ue,\nabla v)+b(\ue,\ue,v)=\big((a\phie-J\ast\phie)\nabla\phie,v\big)+\langle h,v\rangle,\label{wsappb2}
\end{align}
for every $\psi\in V$, every $v\in V_{div}$ and for almost
any $t\in(0,T)$.

On account of (A1) we have $|m(s)F''(s)|\leq a$ for every $s\in[-1,1]$ and for some positive constant $a$. This immediately implies that
\begin{align}
&|\me(s)\Fe''(s)|\leq a,\qquad\forall s\in\mathbb{R},\qquad\forall\e\in(0,1).
\label{est28}
\end{align}
% it not difficult to see that
%there exists $\e_1>$ such that $\me\Fe''$ is bounded on $\mathbb{R}$ uniformly for
%all $\e\in(0,\e_1]$. More precisely, we can show that
%\begin{align}
%&\alpha_1\leq\me(s)\Fe''(s)\leq\alpha_2,\qquad\forall s\in\mathbb{R},\qquad\forall\e\in(0,\e_0],
%\end{align}
%where $\alpha_1=\min(0,\overline{m}\sigma)$, $\overline{m}:=\max_{s\in[-1,1]}m(s)$,
%$\sigma:=\beta+\min_{s\in[-1,1]}F_2''+\gamma-\Vert a\Vert_\infty$, and
%$\alpha_2=1+\max_{s\in[-1,1]}\big[m(s)F''(s)\big]$.
Furthermore, by using the continuity of $mF''$ on $[-1,1]$
(cf. (A1)) and \eqref{conve6}, it is not difficult to see that
\begin{align}
&\me(\phie)\Fe''(\phie)\to m(\varphi)F''(\varphi),\qquad\mbox{ a.e. in }\:\:Q_T.\nonumber
\end{align}
Therefore, we obtain
\begin{align}
&\me(\phie)\Fe''(\phie)\to m(\varphi)F''(\varphi),\qquad\mbox{ strongly in }\:\: L^r(Q_T),\qquad\forall r\in[2,\infty).
\label{conve8}
\end{align}
Notice also that \eqref{conve5}, the continuous embedding
$L^\infty(0,T;L^2(\Omega))\cap L^2(0,T;L^6(\Omega))\hookrightarrow L^{10/3}(Q_T)$
and \eqref{conve6} imply that, for $d=3$, we have
\begin{align}
&\phie\to\varphi\qquad\mbox{strongly in }L^s(Q_T),\qquad\mbox{for all }\:\:\:2\leq s<10/3.
\label{conve9}
\end{align}
For $d=2$, by the Gagliardo-Nirenberg inequality the continuous embedding $L^\infty(0,T;L^2(\Omega))\cap L^2(0,T;V)
\hookrightarrow L^{4}(Q_T)$ holds, thus we get
\begin{align}
&\phie\to\varphi\qquad\mbox{strongly in }L^s(Q_T),\qquad\mbox{for all }\:\:\:2\leq s<4.
\label{conve13}
\end{align}
Furthermore, by Lebesgue's theorem we also deduce
\begin{align}
&\me(\phie)\to m(\varphi) \qquad\mbox{ strongly in }\:\: L^s(Q_T),\qquad\mbox{ for all }\:\: 2\leq s<\infty.
\label{conve14}
\end{align}
We can now multipy \eqref{wsappb1}, \eqref{wsappb2} by $\chi,\omega\in C^\infty_0(0,T)$,
respectively, and integrate the resulting identities with respect to time from $0$ and $T$.
The convergences \eqref{conve1}--\eqref{conve7}
and \eqref{conve8}--\eqref{conve14} are enough to pass
to the limit and to deduce that $u$ and $\varphi$
satisfy the variational formulation \eqref{wform1} and \eqref{wform2}.
Observe that in passing to the limit we can  initially take $\psi$
in $C^1(\overline{\Omega})$ and then prove by density that \eqref{wform1}
holds also for all $\psi$ in $V$.

In order to prove that $z:=[u,\varphi]$ is a weak solution according to Definition
\ref{wsdeg}, we are only left to show that \eqref{reg3} and \eqref{reg5} hold.
To this aim let us first see that $\varphi$ satisfies \eqref{phiLq}.
Indeed, this is a consequence of \eqref{est25} and of the fact that,
since $\varphi\in L^2(0,T;V)$, $\varphi$ is measurable with values in $L^6(\Omega)$
when $d=3$, and measurable with values in $L^p(\Omega)$, for all $p\in[2,\infty)$,
when $d=2$. Recalling the improved regularity \eqref{phiLq}, and also on account of \eqref{est25},
it is now immediate to get
\eqref{reg3} and \eqref{reg5} by a comparison argument in the weak
formulation \eqref{wform1}, \eqref{wform2}. In particular, notice that the contribution
from the convective term in \eqref{wform1} can now be estimated as follows
\begin{align}
&|(u\varphi,\nabla\psi)|\leq \Vert u\Vert_{L^3(\Omega)}\Vert\varphi\Vert_{L^6(\Omega)}\Vert\nabla\psi\Vert
\leq C\Vert\nabla u\Vert\Vert\nabla\psi\Vert,\qquad\forall\psi\in V.
\nonumber
\end{align}

Finally, if $d=2$, the energy identity \eqref{energeq}
can be obtained by taking $\psi=\varphi$ and $v=u$
(see \eqref{reg5} and \eqref{reg3})
as test functions in the weak formulation \eqref{wform1} and \eqref{wform2},
respectively, and adding the resulting identities. In the case $d=3$, we choose first $\psi=\varphi$
as test function in \eqref{wform1} (cf. \eqref{reg5})
and integrate the resulting identity to get \eqref{halfid1}.
Then we observe that, due to the last part of the proof of Corollary \ref{cor1}, the approximate
 solution $z_\e:=[\ue,\phie]$ satisfies
 \begin{align}
&\frac{1}{2}\Vert \ue(t)\Vert^2
+\nu\int_0^t\Vert\nabla \ue\Vert^2\leq\frac{1}{2}\Vert u_0\Vert^2
+\int_0^t\int_\Omega\big(a\phie-J\ast\phie\big)\ue\cdot\nabla\phie+\int_0^t\langle h,\ue\rangle d\tau.
\nonumber
\end{align}
 Passing to the limit as $\e\to 0$ in this last inequality and arguing as at the end
 of the proof of Corollary \ref{cor1}, taking advantage of
 \eqref{conve1}, \eqref{conve2}, \eqref{conve5} and \eqref{conve6}, we get \eqref{halfid2}.
 The energy inequality \eqref{energineq} is then obtained by adding together
 \eqref{halfid1} and \eqref{halfid2}.
\end{proof}

It is worth noting that if the mobility degenerates slightly stronger, then
the weak solution of Theorem \ref{Theor2} satisfies a physically more relevant energy inequality
with respect to the energy identity \eqref{energeq}.
This is proven in the following

\begin{cor}
Let all assumptions of Theorem \ref{Theor2} be satisfied
and $d=2,3$. In addition, assume that \color{black}
%$m'(1)=m'(-1)=0$.
$m(\pm 1)=0$ {\itshape with order $\geq 2$}.
\color{black}
Then the weak solution $z=[u,\varphi]$
given by Theorem \ref{Theor2} also fulfills the following integral inequality
\begin{align}
&\mathcal{E}\big(z(t)\big)+\int_0^t\Big(\nu\Vert\nabla
u\Vert^2+\Big\Vert\frac{\mathcal{J}}{\sqrt{m(\varphi)}}\Big\Vert^2\Big)d\tau\leq\mathcal{E}(z_0)+\int_0^t\langle
h,u\rangle d\tau,\label{otherei}
\end{align}
for all $t>0$, where the mass flux $\mathcal{J}$ given by
\begin{align}
&\mathcal{J}=-m(\varphi)\nabla(a\varphi-J\ast\varphi)-m(\varphi)F''(\varphi)\nabla\varphi\nonumber
\end{align}
is such that
\begin{align}
&\mathcal{J}\in
L^2(Q_T),\qquad\frac{\mathcal{J}}{\sqrt{m(\varphi)}}\in
L^2(Q_T).\nonumber
\end{align}
\end{cor}

\begin{proof}
Setting $\widehat{\mathcal{J}}_\e:=-\sqrt{m_\e(\phie)}\nabla\mue$,
from \eqref{estep3} we have
$\Vert\widehat{\mathcal{J}}_\e\Vert_{L^2(Q_T)}\leq C$ and
therefore there exists $\widehat{\mathcal{J}}\in L^2(Q_T)$ such
that
\begin{align}
&\widehat{\mathcal{J}}_\e\rightharpoonup\widehat{\mathcal{J}}\qquad\mbox{weakly
in }\:\: L^2(Q_T).\label{conve10}
\end{align}
On the other hand, setting
$\mathcal{J}_\e:=-m_\e(\varphi_\e)\nabla\mue=\sqrt{m_\e(\phie)}\widehat{\mathcal{J}}_\e$,
we have also $\Vert\mathcal{J}_\e\Vert_{L^2(Q_T)}\leq C$. Thus
there exists $\mathcal{J}\in L^2(Q_T)$ such that
\begin{align}
&\mathcal{J}_\e\rightharpoonup\mathcal{J}\qquad\mbox{weakly in
}\:\: L^2(Q_T).\label{conve11}
\end{align}
Since $m_\e(\varphi_\e)\to m(\varphi)$ strongly in $L^r(Q_T)$ for
every $r\in[1,\infty)$, from \eqref{conve10} we have
$\sqrt{m_\e(\phie)}\widehat{\mathcal{J}}_\e\rightharpoonup\sqrt{m(\varphi)}\widehat{\mathcal{J}}$
weakly in $L^{2-\gamma}(Q_T)$ for every $\gamma\in(0,1]$ and
therefore by comparison with \eqref{conve11} we get
\begin{align}
&\mathcal{J}:=\sqrt{m(\varphi)}\widehat{\mathcal{J}}.\nonumber
\end{align}
Let us observe now that inequality
\eqref{eiappeps} can be rewritten in the form
\begin{align}
&\mathcal{E}_\e\big(\ze(t)\big)+\int_0^t\Big(\nu\Vert\nabla
\ue\Vert^2+\Big\Vert\frac{\mathcal{J}_\e}{\sqrt{\me(\phie)}}\Big\Vert^2\Big)d\tau\leq\mathcal{E}(z_0)+\int_0^t\langle
h,\ue\rangle d\tau,\qquad\forall t>0,
\label{othereiapp}
\end{align}
where $\mathcal{J}_\e$ is defined in terms of $\phie$, that is,
\begin{align}
&\mathcal{J}_\e=-\me(\phie)\nabla\mue=-\me(\phie)\nabla\big(a\phie-J\ast\phie\big)-\me(\phie)\Fe''(\phie)\nabla\phie.
\nonumber
\end{align}
Using \eqref{conve5}, \eqref{conve8} and \eqref{conve14} it is easy to see that
\begin{align}
&\mathcal{J}_\e\rightharpoonup
-m(\varphi)\nabla\big(a\varphi-J\ast\varphi\big)-m(\varphi)F''(\varphi)\nabla\varphi
\qquad\mbox{weakly in}\:\: L^{2-\eta}(Q_T),\quad\forall\eta\in(0,1).\nonumber
\end{align}
On account of \eqref{conve11}, we therefore see that $\mathcal{J}$ is given by
\begin{align}
&\mathcal{J}=-m(\varphi)\nabla\big(a\varphi-J\ast\varphi\big)-m(\varphi)F''(\varphi)\nabla\varphi,
\label{massflux}
\end{align}
and that the  last convergence \color{black} is also in $L^2(Q_T)$.
Hence, taking the inferior limit
of the third term on the left hand side of \eqref{othereiapp}, we have
\begin{align}
\int_0^t\Big\Vert\frac{\mathcal{J}}{\sqrt{m(\varphi)}}\Big\Vert^2 d\tau
=\int_0^t\Vert\widehat{\mathcal{J}}\Vert^2 d\tau
\leq\liminf_{\e\to 0}\int_0^t\Vert\widehat{\mathcal{J}}_\e\Vert^2 d\tau
=\liminf_{\e\to 0}\int_0^t\Big\Vert\frac{\mathcal{J}_\e}{\sqrt{\me(\phie)}}\Big\Vert^2 d\tau.
\label{conve15}
\end{align}
On the other hand, \color{black} the strong degeneracy condition $m(\pm 1)=0$ with order $\geq 2$ implies
\color{black}
(see {\color{black} Remark~\ref{purebis} here below})
that the set $\{x\in\Omega:|\varphi(x,t)|=1\}$ has zero
measure for almost any $t\in(0,T)$. Hence we have $\Fe(\phie)\to
F(\varphi)$ almost everywhere in $Q_T$, and {\color{black} Fatou's lemma yields}
\begin{align}
&\int_\Omega F(\varphi)\leq\liminf_{\e\to 0}\int_\Omega\Fe(\phie).\label{conve16}
\end{align}
Finally, \eqref{otherei} follows immediately from \eqref{othereiapp} on account of \eqref{conve15}, \eqref{conve16}, \eqref{massflux} and the second weak convergence in \eqref{conve1}.
\end{proof}

\begin{oss}
\label{pure}
{\upshape
In comparison with the analogous result for the case of constant mobility (see \cite[Theorem 1]{FG2}),
Theorem \ref{Theor2} does not require the condition $|\overline{\varphi}_0|<1$
(the assumptions on $\varphi_0$ imply only that $|\overline{\varphi}_0|\leq 1$).
This is essentially due to the fact that here we are dealing with a different weak
formulation.
Therefore, if $F$ is bounded (e.g. $F$ is given by \eqref{potlog}) and at $t=0$
the fluid is in a pure phase, e.g., $\varphi_0=1$ almost everywhere in $\Omega$,
and furthermore $u_0=u(0)$ is given in $G_{div}$, then we can immediately check
that the couple $[u,\varphi]$ given by
$$u=u(\cdot,t),\qquad \varphi=\varphi(\cdot,t)=1,\qquad\mbox{ a.e. in } \Omega,\quad\mbox{a.a. }t\geq 0,$$
where $u$ is solution of the Navier-Stokes equations with non-slip boundary condition,
initial velocity field $u_0$ and external force $h$,
explicitly satisfies the weak formulation \eqref{wform1}, \eqref{wform2}.
Hence, the nonlocal Cahn-Hilliard-Navier-Stokes model with degenerate mobility
and bounded double-well potentials allows pure phases at any time $t\geq 0$.
This possibility is excluded in the model with constant mobility since
in such model the chemical potential $\mu$ (and hence $F'(\varphi)$)
appears explicitly.
}
\end{oss}

\begin{oss}
\label{purebis}
{\upshape
If
%$m'(1)\neq 0$ and $m'(-1)\neq 0$,
\color{black}
$m(\pm 1)=0$ {\itshape with order strictly less than $2$},
\color{black}
then, as a consequence of
(A1) and of the definition of the function $M$, both $F$ and $M$
are bounded in $[-1,1]$. This can be seen, for instance, by writing
$M$ as follows
$$M(s)=\int_0^{\color{black}s\color{black}}\frac{s-t}{m(t)}dt,\qquad\forall s\in(-1,1).$$
In this case, the conditions $F(\varphi_0)\in L^1(\Omega)$ and
$M(\varphi_0)\in L^1(\Omega)$ of Theorem \ref{Theor2} are
satisfied by every initial datum $\varphi_0$ such that
$|\varphi_0|\leq 1$ in $\Omega$. Therefore the
existence of pure phases is allowed. A relevant example for this
case is given by \eqref{potlog}.
%$m(s)=k(1-s^2)$ and
%$F(s)=-(\theta_c/2)s^2+(\theta/2)\big((1+s)\log(1+s)+(1-s)\log(1-s)\big)$.
On the other hand, if \color{black} $m(\pm 1)=0$ {\itshape with order greater than or equal $2$}
\color{black} (in this case we say that $m$ is {\itshape
strongly degenerate}), then it can be proved (cf.
\cite[Corollary]{EG} and also \cite[Theorem 2.3]{B})
that the conditions $F(\varphi_0)\in L^1(\Omega)$ and
$M(\varphi_0)\in L^1(\Omega)$ imply that the sets
$\{x\in\Omega:\varphi_0(x)=1\}$ and
$\{x\in\Omega:\varphi_0(x)=-1\}$ have both measure zero. Hence, in
this case we have obviously $|\overline{\varphi}_0|<1$ and
furthermore it can be seen that also the sets
$\{x\in\Omega:\varphi(x,t)=1\}$ and
$\{x\in\Omega:\varphi(x,t)=-1\}$ have both measure zero for almost any
$t>0$. Therefore, in the case of strongly degenerating mobility
the presence of pure phases is not allowed (even on subsets of
$\Omega$ of positive measure), and so, from this point of view the
situation is more similar to the case of constant
mobility. Summing up, as far as the possibility of existence of
pure phases is concerned, the difference between the cases of
constant and degenerate mobility is more relevant when the
mobility is degenerate, but not strongly degenerate
(on this issue see also \cite{SZ}).\color{black}}
\end{oss}

If, in addition to the assumptions of Theorem \ref{Theor2}, we suppose that
the initial chemical potential $\mu(0)=:\mu_0$ belongs to $H$, then we can improve
the regularity of $\mu$. In particular, we obtain $\mu\in L^2(0,T;V)$.
This is shown in the next theorem, which also requires some further reasonable conditions
on the singular part $F_1$ of the potential $F$ (the potential \eqref{potlog} is however
included) and whose proof is based on a suitable choice of the test function (see \cite{GZ}).
We point out that, while in \cite{GZ} the improved regularity for the chemical potential is
justified only by formal computations, the argument in the proof of Theorem \ref{Theor3}
is rigorous.

\begin{thm}\label{Theor3}
\color{black}
%Assume that all assumptions of Theorem \ref{Theor2} are satisfied
%with $J$ such that
%\begin{align}
%& N_d:=\Big(\sup_{x\in\Omega}\int_\Omega|\nabla J(x-y)|^{\kappa} dy\Big)^{1/\kappa}<\infty,\nonumber
%\end{align}
%where $\kappa=6/5$ if $d=3$ and $\kappa>1$ if $d=2$.
Let all assumptions of Theorem \ref{Theor2} be satisfied.
\color{black}
In addition, assume that $F_1\in C^3(-1,1)$ and that
there exist some constants $\alpha_0,\beta_0>0$ and $\rho\in [0,1)$ such that
the following conditions are fulfilled
%In addition, assume that the following conditions are satisfied
\begin{align}
& \rho F_1''(s)+F_2''(s)+a(x)\geq 0,\qquad\forall s\in(-1,1),\quad\mbox{a.e. in }\Omega,\label{cond-1}\\
&m(s) F_1''(s)\geq\alpha_0,\qquad\forall s\in [-1,1],\label{cond0}\\
&|m^2(s) F_1'''(s)|\leq\beta_0,\qquad \forall s \in [-1,1],\label{cond1}\\
& F_1'(s) F_1'''(s)\geq 0,\qquad\forall s\in (-1,1).\label{cond2}
\end{align}
%Suppose also that $F_2$ is given by
%$$F_2(x,s)=-\frac{1}{2}a(x)s^2.$$
Let $\varphi_0$ be such that
\begin{align}
\label{cond2bis}
&F'(\varphi_0)\in H.
\end{align}
Then, the weak solution $z=[u,\varphi]$ given by Theorem \ref{Theor2} fulfills
\begin{align}
&\mu\in L^\infty(0,T;H),\qquad\nabla\mu\in L^2(0,T;H).
\label{muest}
\end{align}
As a consequence, $z=[u,\varphi]$ also satisfies the weak formulation \eqref{wdefnd1} and \eqref{wdefnd2},
the energy inequality \eqref{ei} and, for $d=2$, the energy identity \eqref{eniden}.
\end{thm}
\begin{oss}
{\upshape {\color{black} Note} that in Theorem \ref{Theor3} assumptions (A2) and (A4) can
be omitted. Indeed, (A2) follows from \eqref{cond0} and (A1), while (A4)
follows from \eqref{cond-1} and \eqref{cond0}. {\color{black} Observe} moreover that the standard
potential $F$ and the mobility $m$ in \eqref{potlog} comply with the assumptions of Theorem \ref{Theor3}.}
\end{oss}

\color{black}
\begin{oss}
{\upshape
Assumption \eqref{cond0} together with (A1) imply that the two conditions $F(\varphi_0)\in L^1(\Omega)$ and $M(\varphi_0)\in L^1(\Omega)$
are equivalent. Indeed, by combining \eqref{cond0} with the definition of function $M$ we get
$\alpha_0 M''(s)\leq F_1''(s)$, for all $s\in(-1,1)$.
Integrating this inequality we obtain
$$\alpha_0 M(s)\leq F_1(s)-F_1(0)-F_1'(0)s,\qquad\forall s\in(-1,1),$$
and, {\color{black} on account of \eqref{yu}}, we get the above equivalence.
}
\end{oss}
\color{black}

\begin{proof}

Let us write the weak formulation of \eqref{Pbe1} as
\begin{align}
&\langle\phie',\psi\rangle+\big(\me(\phie)\nabla\widetilde{\mue},\nabla\psi\big)
+\big(\me(\phie)\nabla(a\phie+F_{2\e}'(\phie)),\nabla\psi\big)=(\ue\phie,\nabla\psi),
\label{wfeps}
\end{align}
for all $\psi\in V$, where we have set
$$\widetilde{\mue}:=w_\e+\Fie'(\phie),\qquad w_\e:=-J\ast\phie,$$
and where $\Fie$ is defined as in \eqref{approxpot1}.
Take $\psi=\Fie'(\phie)\Fie''(\phie)$ as test function in \eqref{wfeps}
and notice that $\psi\in V$. Indeed, $\Fie'(\phie)\in V$ and, since by \eqref{approxpot1}
$\Fie''$ is globally Lipschitz on $\mathbb{R}$, we have also $\Fie''(\phie)\in V$.
On the other hand,
\color{black}
by applying an Alikakos' iteration argument as in \cite[Theorem 2.1]{BH1},
we can prove that the family of approximate solutions $\phie$ is uniformly bounded in $L^\infty(\Omega)$.
To see this,
%Let us now show that the family of $\phie$ is uniformly bounded in $L^\infty(\Omega)$.
let us take $\psi=|\phie|^{p-1}\phie$ as test function in \eqref{wfeps}, where $p>1$.
Then we get the following differential identity
\begin{align}
&\frac{1}{p+1}\frac{d}{dt}\Vert\phie\Vert_{L^{p+1}(\Omega)}^{p+1}+p\int_\Omega m_\e(\phie)\Fie''(\phie)|\nabla\phie|^2|\phie|^{p-1}
+p\int_\Omega m_\e(\phie)\nabla w_\e\cdot\nabla\phie|\phie|^{p-1}\nonumber\\
&+p\int_\Omega m_\e(\phie) \nabla\big(a\phie+F_{2\e}'(\phie)\big)\cdot\nabla\phie|\phie|^{p-1}=0.
\label{phiebdd1}
\end{align}
Actually, the above choice of test function would not be generally admissible. Nevertheless, the argument
%leading to\eqref{phiebdd1}
can be made rigorous by means of a density procedure, e.g., by first truncating the test
function $|\phie|^{p-1}\phie$ and then passing to the limit with respect to the truncation parameter.

By using \eqref{cond-1} we obtain
\begin{align}
&p\int_\Omega m_\e(\phie)\Fie''(\phie)|\nabla\phie|^2|\phie|^{p-1}
+p\int_\Omega m_\e(\phie) \nabla\big(a\phie+F_{2\e}'(\phie)\big)\cdot\nabla\phie|\phie|^{p-1}\nonumber\\
&=p\int_\Omega m_\e(\phie)\Fie''(\phie)|\nabla\phie|^2|\phie|^{p-1}+
p\int_\Omega m_\e(\phie) \big(a+F_{2\e}''(\phie)\big)|\nabla\phie|^2|\phie|^{p-1}\nonumber\\
&+p\int_\Omega m_\e(\phie) \phie\nabla a\cdot\nabla\phie|\phie|^{p-1}
\geq (1-\rho)p\int_\Omega m_\e(\phie)\Fie''(\phie)|\nabla\phie|^2|\phie|^{p-1}\nonumber\\
&+p\int_\Omega m_\e(\phie) \phie\nabla a\cdot\nabla\phie|\phie|^{p-1}
\geq\frac{4\alpha_0 p(1-\rho)}{(p+1)^2}\int_\Omega\Big|\nabla\big|\phie\big|^{\frac{p+1}{2}}\Big|^2\nonumber\\
&+p\int_\Omega m_\e(\phie) \phie\nabla a\cdot\nabla\phie|\phie|^{p-1}.\label{phiebdd2}
\end{align}
Therefore, by combining \eqref{phiebdd1} with \eqref{phiebdd2}, we deduce
\begin{align}
&\frac{1}{p+1}\frac{d}{dt}\Vert\phie\Vert_{L^{p+1}(\Omega)}^{p+1}+\frac{4\alpha_0 p(1-\rho)}{(p+1)^2}\int_\Omega\Big|\nabla\big|\phie\big|^{\frac{p+1}{2}}\Big|^2
+p\int_\Omega m_\e(\phie)\nabla w_\e\cdot\nabla\phie|\phie|^{p-1}\nonumber\\
&+p\int_\Omega m_\e(\phie) \phie\nabla a\cdot\nabla\phie|\phie|^{p-1}\leq 0.
\label{phiebdd3}
\end{align}
Starting from \eqref{phiebdd3} \color{black} and using the fact that $m_\e(\phie)$
  is uniformly bounded with respect to $\e$, \color{black} we can argue exactly as in \cite[Proof of Theorem 2.1]{BH1} in order to conclude that
\begin{align}
&\Vert\phie\Vert_{L^\infty(\Omega)}\leq C,\label{unifboundphie}
%(\Vert\varphi_0\Vert_{L^\infty(\Omega)}),
\end{align}
where the positive constant $C$ depends on $\Vert\varphi_0\Vert_{L^\infty(\Omega)}$ and on $F$, $J$, $\Omega$, but is
independent of $\e$.

\color{black}

 Hence,
 \color{black}
 since $\phie\in L^\infty(\Omega)$, then
 \color{black}
  $\Fie'(\phie)$ and $\Fie''(\phie)$
are in $V\cap L^\infty(\Omega)$. This yields $\psi=\Fie'(\phie)\Fie''(\phie)\in V\cap L^\infty(\Omega)$.
Using this test function we first see that the contribution of the convective term vanishes.
Indeed, by the incompressibility condition \eqref{sy4} we deduce
\begin{align}
&\int_\Omega(\ue\cdot\nabla\phie)\Fie'(\phie)\Fie''(\phie)=\int_\Omega u_\e\cdot\nabla\Big(\frac{\Fie'^2(\phie)}{2}\Big)=0.\nonumber
\end{align}
Furthermore, we can see that we have
\begin{align}
&\langle\phie',\Fie'(\phie)\Fie''(\phie)\rangle=\frac{1}{2}\frac{d}{dt}\int_\Omega|\Fie'(\phie)|^2.
\label{chainrule}
\end{align}
Indeed, setting $G_\e(s):=\Fie'^2(s)/2$, we find $G_\e''=\Fie'\Fie'''+\Fie''^2\geq 0$
almost everywhere in $\mathbb{R}$,
due to \eqref{cond2}, i.e., $G_\e$ is convex on $\mathbb{R}$.
Hence, \eqref{chainrule} follows as an easy application of the chain rule result proved in
\cite[Proposition 4.2]{CKRS}.

We are therefore led to the following identity
\begin{align}
&\frac{1}{2}\frac{d}{dt}\Vert\Fie'(\phie)\Vert^2+\int_\Omega \me(\phie)\Fie''(\phie)
\nabla\widetilde{\mue}\cdot\nabla\Fie'(\phie)
\nonumber\\
&+\int_\Omega\me(\phie)\Fie'(\phie)\Fie'''(\phie)\nabla\widetilde{\mue}\cdot\nabla\phie
+\int_\Omega\me(\phie)\phie\nabla a\cdot\nabla\big(\Fie'(\phie)\Fie''(\phie)\big)\nonumber\\
&+\int_\Omega\me(\phie)\big(a+F_{2\e}''(\phie)\big)\nabla\phie\cdot\nabla\big(\Fie'(\phie)\Fie''(\phie)\big)
=0,
%\label{ideps}
\end{align}
%By setting
%$$\mue:=w_\e+\Fie'(\phie),\qquad w_\e:=a\phie-J\ast\phie+\overline{F}_2'(\phie)=-J\ast\phie,$$
%where we have taken $\overline{F}_2(s)=-as^2/2$.
%Then \eqref{ideps}
which can be rewritten as
\begin{align}
&\frac{1}{2}\frac{d}{dt}\Vert\widetilde{\mue}-w_\e\Vert^2+
\int_\Omega \me(\phie)\Fie''(\phie)|\nabla(\widetilde{\mue}-w_\e)|^2+
\int_\Omega \me(\phie)\Fie''(\phie)\nabla w_\e\cdot\nabla(\widetilde{\mue}-w_\e)\nonumber\\
&+\int_\Omega\me(\phie)(\widetilde{\mue}-w_\e)\Fie'''(\phie)\nabla\widetilde{\mue}\cdot\nabla\phie
+\int_\Omega\me(\phie)\phie\Fie''(\phie)\nabla a\cdot\nabla\Fie'(\phie)\nonumber\\
&+\int_\Omega\me(\phie)\phie\Fie'(\phie)\Fie'''(\phie)\nabla a\cdot\nabla\phie+
\int_\Omega\me(\phie)\big(a+F_{2\e}''(\phie)\big)\Fie''^2(\phie)|\nabla\phie|^2\nonumber\\
&+\int_\Omega\me(\phie)\big(a+F_{2\e}''(\phie)\big)\Fie'(\phie)\Fie'''(\phie)|\nabla\phie|^2=0.
\label{ideps2}
\end{align}
Condition \eqref{cond-1} implies that
\begin{align}\label{ideps2bis}
&\rho\Fie''(s)+F_{2\e}''(s)+a(x)\geq 0,\qquad\forall s\in\mathbb{R},\quad\mbox{ a.a }x\in\Omega, \quad
 \forall \e\in (0,1).
\end{align}
Then we get
\begin{align}
&\int_\Omega \me(\phie)\Fie''(\phie)|\nabla(\widetilde{\mue}-w_\e)|^2+
\int_\Omega\me(\phie)\big(a+F_{2\e}''(\phie)\big)\Fie''^2(\phie)|\nabla\phie|^2\nonumber\\
& \geq  (1-\rho)\int_\Omega\me(\phie)\Fie''(\phie)|\nabla(\widetilde{\mue}-w)|^2.
\label{ideps3}
\end{align}
On the other hand, condition
\eqref{cond2} implies that, for all $\e\in(0,1)$, $\Fie'\Fie'''\geq 0$ on $\mathbb{R}$.
Thus using again \eqref{ideps2bis}, we have
\begin{align}
&\int_\Omega\me(\phie)(\widetilde{\mue}-w_\e)\Fie'''(\phie)\nabla\widetilde{\mue}\cdot\nabla\phie
+\int_\Omega\me(\phie)\big(a+F_{2\e}''(\phie)\big)\Fie'(\phie)\Fie'''(\phie)|\nabla\phie|^2\nonumber\\
&\geq\int_{\Omega}\me(\phie)\Fie'(\phie)\Fie'''(\phie)\nabla\widetilde{\mue}\cdot\nabla\phie
-\rho\int_\Omega\me(\phie)\Fie''(\phie)\Fie'(\phie)\Fie'''(\phie)|\nabla\phie|^2\nonumber\\
&=\int_{\Omega}\me(\phie)\Fie'(\phie)\Fie'''(\phie)\Fie''(\phie)|\nabla\phie|^2
+\int_\Omega\me(\phie)\Fie'(\phie)\Fie'''(\phie)\nabla w_\e\cdot\nabla\phie\nonumber\\
&-\rho\int_{\Omega}\me(\phie)\Fie'(\phie)\Fie'''(\phie)\Fie''(\phie)|\nabla\phie|^2\nonumber\\
&=(1-\rho)\int_\Omega\me(\phie)\Fie'(\phie)\Fie'''(\phie)\nabla(\widetilde{\mue}-w_\e)
\cdot\nabla\phie+\int_\Omega\me(\phie)\Fie'(\phie)\Fie'''(\phie)\nabla w_\e\cdot\nabla\phie\nonumber\\
&=(1-\rho)\int_\Omega\me(\phie)\Fie'(\phie)\Fie'''(\phie)\nabla\widetilde{\mue}
\cdot\nabla\phie+\rho\int_\Omega\me(\phie)\Fie'(\phie)\Fie'''(\phie)\nabla w_\e\cdot\nabla\phie.
\label{ideps4}
\end{align}
Furthermore, observe that
% conditions \eqref{cond2}, \eqref{cond1} imply that, for $\e$ small enough,
%$\me\Fie''=:\widetilde{\alpha}\in C([-1,1])$ with %$0<\widetilde{\alpha}_0\leq\widetilde{\alpha}(s)\leq\widetilde{\alpha}_1$
%for all $s\in[-1,1]$ and $\me^2\Fie'''=:\widetilde{\beta}\in C([-1,1])$, and therefore
\begin{align}
&\me(\phie)\Fie'''(\phie)\nabla\phie=\frac{\me^2(\phie)\Fie'''(\phie)}{\me(\phie)\Fie''(\phie)}\nabla\Fie'(\phie)
%\frac{\widetilde{\beta}(\phie)}{\me(\phie)}\nabla\phie
%=\widetilde{\beta}(\phie)\frac{\Fie''(\phie)}{\widetilde{\alpha}(\phie)}\nabla\phie
\label{est15}
=\gamma_\e(\phie)\nabla\Fie'(\phie),
\end{align}
where
\begin{align}
&\gamma_\e(s):=\frac{\me^2(s)\Fie'''(s)}{\me(s)\Fie''(s)}, \qquad\forall s\in\mathbb{R}.\nonumber
\end{align}
%\gamma(s)=\widetilde{\beta}(s)/\widetilde{\alpha}(s)$.
Therefore the following identity holds
\begin{align}
&\int_\Omega\me(\phie)\phie\Fie'(\phie)\Fie'''(\phie)\nabla a\cdot\nabla\phie
=\int_\Omega\gamma_\e(\phie)(\widetilde{\mue}-w_\e)\phie\nabla a\cdot\nabla(\widetilde{\mue}-w_\e).
\label{ideps5}
\end{align}
By plugging \eqref{ideps3}, \eqref{ideps4} and \eqref{ideps5}
into \eqref{ideps2} and using \eqref{est15} once more, we get
\begin{align}
&\frac{1}{2}\frac{d}{dt}\Vert\widetilde{\mue}-w_\e\Vert^2+(1-\rho)
\int_\Omega \me(\phie)\Fie''(\phie)|\nabla(\widetilde{\mue}-w_\e)|^2\nonumber\\
&+\int_\Omega \me(\phie)\Fie''(\phie)\nabla w_\e\cdot\nabla(\widetilde{\mue}-w_\e)
+(1-\rho)\int_\Omega\gamma_\e(\phie)(\widetilde{\mue}-w_\e)|\nabla(\widetilde{\mue}-w_\e)|^2\nonumber\\
&+\int_\Omega\gamma_\e(\phie)(\widetilde{\mue}-w_\e)\nabla(\widetilde{\mue}-w_\e)\cdot\nabla w_\e
+\int_\Omega\me(\phie)\Fie''(\phie)\phie\nabla a\cdot\nabla(\widetilde{\mue}-w_\e)\nonumber\\
&+\int_\Omega\gamma_\e(\phie)(\widetilde{\mue}-w_\e)\phie\nabla a\cdot\nabla(\widetilde{\mue}-w_\e)
\leq 0. \label{th0}
\end{align}
Notice that the fourth term on the left hand side is nonnegative since there holds
%thanks to
%condition
%\eqref{cond2} which implies that for all $\e\in(0,1)$ $\Fie'\Fie'''\geq 0$ on $\mathbb{R}$,
\begin{align}
&\gamma_\e(\phie)(\widetilde{\mue}-w_\e)=\gamma(\phie)\Fie'(\phie)
=\frac{\me^2(\phie)\Fie'''(\phie)}{\me(\phie)\Fie''(\phie)}\Fie'(\phie)\geq 0.\nonumber
%\label{est16}
\end{align}
On the other hand, condition \eqref{cond0} implies that
\begin{align}
&\me(s)\Fie''(s)\geq\alpha_0,\qquad\forall s\in\mathbb{R},
\qquad\forall\e\in(0,1), \label{est33}
\end{align}
and, using (H2), we have
\begin{align}
&\int_\Omega\gamma_\e(\phie)(\widetilde{\mue}-w_\e)\nabla(\widetilde{\mue}-w_\e)\cdot\nabla w_\e
\geq-(1-\rho)\frac{\alpha_0}{8}\Vert\nabla(\widetilde{\mue}-w_\e)\Vert^2\nonumber\\
&-C_1\int_\Omega\gamma_\e^2(\phie)|\nabla w_\e|^2|\widetilde{\mue}-w_\e|^2,\label{th1}\\
&\int_\Omega\gamma_\e(\phie)(\widetilde{\mue}-w_\e)\phie\nabla a\cdot\nabla(\widetilde{\mue}-w_\e)
\geq-(1-\rho)\frac{\alpha_0}{8}\Vert\nabla(\widetilde{\mue}-w_\e)\Vert^2\nonumber\\
&-C_1\int_\Omega\gamma_\e^2(\phie)\phie^2|\nabla a|^2|\widetilde{\mue}-w_\e|^2,\label{th}\\
&\int_\Omega\me(\phie)\Fie''(\phie)\phie\nabla a\cdot\nabla(\widetilde{\mue}-w_\e)
\geq-(1-\rho)\frac{\alpha_0}{8}\Vert\nabla(\widetilde{\mue}-w_\e)\Vert^2
-C_2\Vert\phie\Vert^2,\label{th2}\\
%&\int_\Omega\gamma_\e(\phie)(\mue-w_\e)\nabla(\mue-w_\e)\cdot\nabla w_\e
%\geq-(1-\rho)\frac{\alpha_0}{12}\Vert\nabla(\mue-w_\e)\Vert^2\nonumber\\
%&-\frac{2}{\alpha_0}\int_\Omega\gamma_\e^2(\phie)|\nabla w_\e|^2|\mue-w_\e|^2,\\
&\int_\Omega \me(\phie)\Fie''(\phie)\nabla w_\e\cdot\nabla(\widetilde{\mue}-w_\e)
\geq-(1-\rho)\frac{\alpha_0}{8}\Vert\nabla(\widetilde{\mue}-w_\e)\Vert^2
-C_3\Vert\nabla w_\e\Vert^2.\label{th3}
\end{align}
In \eqref{th2} and \eqref{th3} we have used the fact that \eqref{est28} holds as a consequence of the
boundedness of $mF''$ on $[-1,1]$.

%\noindent
%\color{black}
%Therefore, the last term on the left hand side of \eqref{ideps2} can be rewritten in the form
%\begin{align}
%&\int_\Omega\gamma_\e(\phie)(\mue-w_\e)\nabla\mue\cdot\nabla(\mue-w_\e).
%\label{lasterm}
%\end{align}

Recall that $\gamma_\e(s)=0$ for all $|s|>1-\e$, and $\gamma_\e(s)=m^2(s)F_1'''(s)/m(s)F_1''(s)$
for $|s|\leq 1-\e$.
Hence, on account of \eqref{cond0} and \eqref{cond1}, we have
\begin{align}
&|\gamma_\e(s)|\leq\gamma_0,\qquad\forall
s\in\mathbb{R},\qquad\forall\e\in(0,1),\label{est34}
\end{align}
with $\gamma_0:=\beta_0/\alpha_0$.
Thus the second term on the right hand side of \eqref{th}
can be estimated as follows
\begin{align}
&%\int_\Omega\gamma_\e^2(\phie)\phie^2|\nabla a|^2|\widetilde{\mu}-w_\e|^2
\int_{\{|\phie|\leq 1-\e\}}\gamma_\e^2(\phie)\phie^2|\nabla a|^2|\widetilde{\mu}-w_\e|^2
\leq\gamma_0^2\Vert\nabla a\Vert_{L^\infty(\Omega)}^2\Vert\widetilde{\mue}-w_\e\Vert^2. \label{th4}
\end{align}
%\noindent
%\color{black}
%\color{black}
%On the other hand, condition
%\eqref{cond2} implies that, for all $\e\in(0,1)$, $\Fie'\Fie'''\geq 0$ on $\mathbb{R}$, so that
%\begin{align}
%&\gamma_\e(\phie)(\mue-w_\e)=\gamma(\phie)\Fie'(\phie)
%=\frac{\me^2(\phie)\Fie'''(\phie)}{\me(\phie)\Fie''(\phie)}\Fie'(\phie)\geq 0.\label{est16}
%\end{align}
%Hence, using \eqref{est15} and \eqref{est16}, \eqref{lasterm} can be estimated from below as follows
%\begin{align}
%&%\int_\Omega\me(\phie)(\mue-w_\e)\Fie'''(\phie)\nabla\mue\cdot\nabla\phie
%\int_\Omega\gamma_\e(\phie)(\mue-w_\e)\nabla\mue\cdot\nabla(\mue-w_\e)\geq
%\int_\Omega\gamma_\e(\phie)(\mue-w_\e)\nabla(\mue-w_\e)\cdot\nabla w_\e\nonumber\\
%&\geq-\frac{\alpha_0}{4}\Vert\nabla(\mue-w_\e)\Vert^2
%-\frac{1}{\alpha_0}\int_\Omega\gamma_\e^2(\phie)|\nabla w_\e|^2|\mue-w_\e|^2.\nonumber
%\end{align}
%By also estimating from below the third term on the left hand side of \eqref{ideps2},
%we obtain (see also \eqref{est28})
%\begin{align}
%&\int_\Omega \me(\phie)\Fie''(\phie)\nabla w_\e\cdot\nabla(\mue-w_\e)
%\geq-\frac{\alpha_0}{4}\Vert\nabla(\mue-w_\e)\Vert^2
%-\frac{a^2}{\alpha_0}\Vert\nabla w_\e\Vert^2,\nonumber
%\end{align}
By means of {\color{black} \eqref{unifboundphie} and \eqref{est33}--\eqref{th4},
we deduce from} \eqref{th0} the inequality
\begin{align}
&\frac{d}{dt}\Vert\widetilde{\mue}-w_\e\Vert^2+\alpha_0(1-\rho)\Vert\nabla(\widetilde{\mue}-w_\e)\Vert^2
\color{black}
\leq C_4+C_5\Vert\widetilde{\mue}-w_\e\Vert^2.
\color{black}
\label{est17}
\end{align}

It is easy to see that
% due to (A3) and (A4) we have (cf. \cite[Proof of Theorem 1]{FG2})
\begin{align}
&|\Fie'(s)|\leq |F_1'(s)|,\qquad\forall s\in(-1,1),\qquad\forall\e\in(0,\e_0],
\label{est35}
\end{align}
with $\e_0>0$ small enough. Therefore, condition \eqref{cond2bis} entails
$$\Vert\widetilde{\mue}(0)-w_\e(0)\Vert=\Vert\Fie'(\varphi_0)\Vert\leq \Vert F_1'(\varphi_0)\Vert,$$
and by applying Gronwall's lemma to \eqref{est17} and using the fact that $w_\e$ is bounded independently of $\e$ in $L^\infty(0,T; H)$ as well as its gradient (cf. (H2) and \eqref{reg4}),
we immediately get
\begin{align}
&\widetilde{\mu}\in L^\infty(0,T;H),\qquad\nabla\widetilde{\mu}\in L^2(0,T;H),\label{sd1}
\end{align}
where $\widetilde{\mu}:=w+\color{black} F_1'(\varphi)\color{black}$ and $w:=-J\ast\varphi$.
Furthermore, due to \eqref{approxpot2},
we have (see also \eqref{F2eps})
\begin{align}
&|F_{2\e}'(s)|\leq L_3|s|+L_4,\qquad |F_{2\e}''(s)|\leq L_5,\qquad\forall s\in\mathbb{R},\nonumber
\end{align}
for some nonnegative constants $L_3,L_4,L_5$ which are independent of $\e$. By using these estimates, \eqref{estep2} and
\eqref{estep4} we also obtain
\begin{align}
& F_2'(\varphi)\in L^\infty(0,T;H),\qquad\nabla F_2'(\varphi)\in L^2(0,T;H).\label{sd2}
\end{align}
From \eqref{sd1} and \eqref{sd2} we immediately get
 \eqref{muest}.

In order to see that \eqref{wdefnd1} and \eqref{wdefnd2} are satisfied, consider
the weak formulation \eqref{wsappb1}, \eqref{wsappb2} of the approximate problem P$_\e$,
write $\Fe''(\phie)\nabla\phie=\nabla\Fe'(\phie)$, rearrange the terms on the left hand
side of \eqref{wsappb1}, take the definition of $\mue$ into account and pass to the limit as $\e\to 0$.

The energy inequality \eqref{ei} can be deduced by passing to the limit in \eqref{eiappeps} and
by observing that
$$\int_\Omega F_\e(\varphi_0)\leq\int_\Omega F_1(\varphi_0)+\int_\Omega F_{2\e}(\varphi_0),$$
and that $\int_\Omega F_{2\e}(\varphi_0)\to\int_\Omega F_{2}(\varphi_0)$
by Lebesgue's theorem (see \eqref{F2eps}).

Finally, the energy identity \eqref{eniden} for $d=2$ is obtained by choosing $\psi=\mu$
as test function in \eqref{wdefnd1}.
\end{proof}

The argument of Theorem \ref{Theor3} can be exploited to prove rigorously
the existence of a weak solution satisfying \eqref{phiLq} with $6<p<\infty$,
for $d=3$.
We point out that the $L^\infty(L^p)$ regularity of $\varphi$ follows only formally
from the condition \color{black} $|\varphi|\leq 1$ \color{black} proved in Theorem \ref{Theor2}. Indeed, in the case $d=3$
it is not known that the $\varphi$ component of the weak solution of Theorem \ref{Theor2}
is measurable with values in $L^p$ if $p>6$. Instead, a rigorous way to deduce such regularity
is by means of an approximation argument which makes use of an approximating potential
with polynomial growth of order $p$.

%Such regularity can therefore be deduced
%In particular we have the following

\begin{cor}
Let all the assumptions of Theorem \ref{Theor2} and Theorem \ref{Theor3}
be satisfied and $d=3$. In particular let $\varphi_0$ be such that \eqref{cond2bis} holds.
In addition, assume
%$F_1\in C^3(-1,1)$ and
that there exists $\e_0>0$ such that
the following assumptions are satisfied
\begin{align}
&F_1'''(s)\geq 1,\quad\forall s\in[1-\e_0,1) \qquad\mbox{ and
}\qquad F_1'''(s)\leq -1,\quad\forall s\in(-1,-1+\e_0],
%\item[(A7)] %There exists $\e_0>0$ such that
% $F_1'''$
%is non-decreasing in $(-1,-1+\e_0]\cup[1-\e_0,1)$,
\label{furthass}\\
&F_1'(s)\geq 0,\quad\forall s\in[1-\e_0,1) \qquad\mbox{ and
}\qquad F_1'(s)\leq 0,\quad\forall s\in(-1,-1+\e_0].\label{furthass2}
\end{align}
Let $p\in (6,\infty)$ be fixed arbitrarily. Then, for every $T>0$, there exists a weak solution
$z=[u,\varphi]$ to \eqref{sy1}--\eqref{sy6} on $[0,T]$  in the sense of Definition \ref{wsdeg}, \color{black}
satisfying the weak formulation \eqref{wdefnd1}, \eqref{wdefnd2} and such that
\begin{align}
&\varphi\in L^\infty(0,T;L^p(\Omega)).
\label{phiLq2}
\end{align}
\end{cor}

\begin{proof}
We argue as in the proofs of Theorem \ref{Theor2} and Theorem \ref{Theor3},
but we employ a different approximation of the singular part $\Fie$ of the double-well potential.
Namely, here we consider the approximate problem \eqref{Pbe1}--\eqref{Pbe6}
with the following choice for $\Fie$
\begin{equation}
\Fie(s)=\left\{\begin{array}{lll}
F_1(1-\e)+F_1'(1-\e)\big(s-(1-\e)\big)+\frac{1}{2}F_1''(1-\e)\big(s-(1-\e)\big)^2\\
+\frac{1}{6}\big(s-(1-\e)\big)^p,\qquad s\geq 1-\e\\
F_1(s),\qquad|s|\leq 1-\e\\
F_1(-1+\e)+F_1'(-1+\e)\big(s-(-1+\e)\big)+\frac{1}{2}F_1''(-1+\e)\big(s-(-1+\e)\big)^2\\
+\frac{1}{6}\big|s-(-1+\e)\big|^p,\qquad s\leq -1+\e.
\end{array}\right.\nonumber
\end{equation}
Let us argue as in Theorem \ref{Theor2} to deduce the estimates
for $\phie$ and $\ue$ (and their time derivatives).
Instead, the passage to the limit will be performed differently with respect to
Theorem \ref{Theor2}.

%All the estimates obtained in the proof of Theorem \ref{Theor2}
%still hold (instead, the passage to the limit will be different).

We first check that \eqref{est26} is still true for this new approximate potential.
Indeed, for $1-\e\leq s<1$, thanks to \eqref{furthass} we have, for some $\xi\in(1-\e,s)$
\begin{align}
F_1(s)&=F_1(1-\e)+F_1'(1-\e)\big(s-(1-\e)\big)+\frac{1}{2}F_1''(1-\e)\big(s-(1-\e)\big)^2\nonumber\\
&+\frac{1}{6}F_1'''(\xi)\big(s-(1-\e)\big)^3 \nonumber\\
&\geq F_1(1-\e)+F_1'(1-\e)\big(s-(1-\e)\big)+\frac{1}{2}F_1''(1-\e)\big(s-(1-\e)\big)^2\nonumber\\
&+\frac{1}{6}\big(s-(1-\e)\big)^p= \Fie(s),\nonumber
\end{align}
provided that $\e\in(0,\e_0]$, with $\e_0$ as in \eqref{furthass}. For $-1<s\leq -1+\e$ we argue similarly.

It is easy to see that the uniform (with respect to $\e$) coercivity condition
\eqref{est20} will now be replaced by (cf. also (A2))
\begin{align}
&\Fe(s)\geq\frac{1}{24}|s|^p-c_p,\qquad\forall s\in\mathbb{R},\qquad\forall\e\in(0,\e_0],
\label{est27}
\end{align}
where the constant $c_p$ depends on $p$ but is independent of $\e$.
It is also immediate to check that \eqref{est21} and (H7) hold for the new approximate potential.
Therefore, since we have also $\Fe\in C^{2,1}_{loc}(\mathbb{R})$, assumptions (H1)--(H6) of Theorem \ref{thm}, with (H4)
replaced by (H7), are satisfied. Hence, for every fixed $\e\in(0,\e_0]$, Problem P$_\e$
admits a weak solution satisfying \eqref{propreg1}, the first of \eqref{propreg2}, \eqref{propreg4}, \eqref{propreg5}, and the energy inequality \eqref{eiappeps}.
By means of \eqref{est22}, \eqref{est21} and \eqref{est27}, from the approximate energy inequality we can still recover estimates
\eqref{estep1}, \eqref{estep3}, \eqref{estep4} and \eqref{estep5}, while
estimate \eqref{estep2} will now be substituted with the stronger
\begin{align}
&\Vert\phie\Vert_{L^\infty(0,T;L^p(\Omega))}\leq C.
\end{align}
The estimates for $\ue'$ and $\phie'$ can be improved as well.
Indeed, instead of \eqref{estep6} and \eqref{estep8} we now find
\begin{align}
&\Vert\phie'\Vert_{L^2(0,T;V')}\leq C,\qquad d=3,
\end{align}
and
\begin{align}
&\Vert\ue'\Vert_{L^2(0,T;V_{div}')}\leq C,\qquad
\Vert\phie'\Vert_{L^2(0,T;V')}\leq C,\qquad d=2,
\end{align}
respectively, while \eqref{estep7} still holds. By compactness results we can still deduce that
there exists a couple $z:=[u,\varphi]$ and a subsequence such that
\eqref{conve1}--\eqref{conve7} hold. Instead, in \eqref{conve4}, \eqref{conve12} and \eqref{conve7}
we shall have $L^2$ in place of $L^{2-\gamma}$, $L^{4/3}$ and $L^{2-\delta}$, respectively,
and
\begin{align}
&\phie\rightharpoonup\varphi\qquad\mbox{weakly}^{\ast}\mbox{ in }L^{\infty}(0,T;L^p(\Omega)),
\quad\mbox{weakly in }L^2(0,T;V),\label{conve17}
\end{align}
in place of \eqref{conve5}.

As far as the passage to the limit is concerned, the situation is now a bit different. Indeed,
due to the polynomial growth of $\Fie$, a strong convergence for the term
$\me(\phie)\Fie''(\phie)$ in some $L^s(Q_T)$ space with $s\geq 2$
is no longer available in order to deduce the weak formulation \eqref{wform1} for $\varphi$.
Therefore, the idea is to use \eqref{cond2bis}, argue as in Theorem \ref{Theor3},
get a control for $\mue$ in $L^2(0,T;V)$ and pass to the limit to deduce
the weak formulation \eqref{wdefnd1} for $\varphi$.
Of course, we have to check that the argument of Theorem \ref{Theor3}
still applies with this different choice of $\Fie$.
Now, it is easy to verify that the test function $\psi=\Fie'(\phie)\Fie''(\phie)$
in the weak formulation of \eqref{Pbe1} is still in $V$.
Indeed, $\Fie''$ is now locally Lipschitz on $\mathbb{R}$
and $\phie\in V\cap L^\infty(\Omega)$, which implies that
$\Fie''(\phie)\in V\cap L^\infty(\Omega)$. Thus $\psi\in V$.

By repeating the calculations in the first part of the proof of
Theorem \ref{Theor3} we are then led again to the differential
identity \eqref{ideps2}. It is immediate to check that
\eqref{est33} still holds. As far as \eqref{est34} is concerned,
notice that in this case we have for, e.g., $s> 1-\e$
\begin{align}
\gamma_\e(s)&=\frac{m^2(1-\e)\frac{1}{6}p(p-1)(p-2)\big(s-(1-\e)\big)^{p-3}}{m(1-\e)\Big(F_1''(1-\e)+\frac{1}{6}p(p-1)\big(s-(1-\e)\big)^{p-2}\Big)}\nonumber\\
&\leq\frac{c_{1,p}m^2(1-\e)\big(s-(1-\e)\big)^{p-3}}{\alpha_0+c_{2,p}m(1-\e)\big(s-(1-\e)\big)^{p-2}},\nonumber
\end{align}
where \eqref{cond0} has been taken into account. Now, if we
consider the function
$$\psi_\delta(\tau):=\frac{\delta^2\tau^{p-3}}{1+\delta\tau^{p-2}},\qquad\forall\tau\geq
0,$$ by an elementary calculation we can see that $\psi_\delta$
has its maximum at
$\tau^\ast:=\big((p-3)/\delta\big)^{\frac{1}{p-2}}$ and that
$\psi_\delta(\tau)\leq\psi_\delta(\tau^\ast)=c_p
\delta^{(p-1)/(p-2)}\leq c_p$, for all $\tau\geq 0$ and for all
$\delta\in(0,1)$, where $c_p$ is a positive constant depending
only on $p$. Hence we deduce that $|\gamma_\e(s)|\leq c_p'$ for
all $|s|>1-\e$ and for all $\e\in (0,1)$. Using also the fact that
$\gamma_\e(s)=m^2(s)F_1'''(s)/m(s)F_1''(s)$ for $|s|\leq 1-\e$,
and taking \eqref{cond0} and \eqref{cond1} into account, we can
therefore conclude that
\begin{align}
&|\gamma_\e(s)|\leq\gamma_{0,p},\qquad \forall
s\in\mathbb{R},\qquad\forall\e\in(0,1),
\end{align}
with $\gamma_{0,p}$ given by
$\gamma_{0,p}=\max(c_p',\beta_0/\alpha_0)$ and independent of
$\e$. Hence, \eqref{est34} is still satisfied.

On the other hand, using (A2) and \eqref{furthass2}, it is easy to check that
 also conditions
 $$\Fie'\Fie'''\geq 0\qquad\mbox{ on }\:\:\mathbb{R},\qquad\forall\e\in(0,1),$$
and \eqref{est35} still hold.
We can therefore conclude that the argument of Theorem \ref{Theor3}
applies and that, due also to \eqref{cond2bis}, the following estimates hold (cf. \eqref{est18})
\begin{align}
&\Vert\mue\Vert_{L^\infty(0,T;H)}\leq C,\qquad\Vert\nabla\mue\Vert_{L^2(0,T;H)}\leq C,
\label{est36}
\end{align}
where the constant $C$ depends on $\Vert\mu_0\Vert$ or, equivalently, on $\Vert F'(\varphi_0)\Vert$.

Finally, let us consider the weak formulation \eqref{wdefnd1} and \eqref{wdefnd2}
for the approximate problem \eqref{Pbe1}--\eqref{Pbe6}. By passing to the limit as $\e\to 0$
and by means of \eqref{est36} and the convergences provided in the proof of Theorem \ref{Theor2},
it is not difficult to show that $z=[u,\varphi]$, which in particular satisfies \eqref{phiLq2}
due to the first of \eqref{conve17}, satisfies the weak formulation \eqref{wdefnd1} and \eqref{wdefnd2}.
\end{proof}

\section{The global attractor in 2D}\setcounter{equation}{0}
\label{sec:long}

In this section we consider system \eqref{sy1}--\eqref{sy6} for $d=2$
and we suppose that the external force $h$ is
time-independent, namely,
\begin{description}
\item[(A5)]  $h\in V_{div}'.$
\end{description}

Let us introduce the set $\mathcal{G}_{m_0}$ of all weak solutions
to \eqref{sy1}--\eqref{sy6} (in the sense
of Definition \ref{wsdeg}) corresponding to all
initial data $z_0=[u_0,\varphi_0]\in\mathcal{X}_{m_0}$, where the
phase space $\mathcal{X}_{m_0}$ is the metric space defined by
\begin{align}
&\mathcal{X}_{m_0}:=G_{div}\times\mathcal{Y}_{m_0},\nonumber
\end{align}
with
$\mathcal{Y}_{m_0}$ given by
\begin{align}
&\label{phasesp2} &\mathcal{Y}_{m_0}:=\big\{\varphi\in
L^{\infty}(\Omega): |\varphi|\leq 1\:\:\mbox{ a.e. in
}\Omega,\:\:\:F(\varphi),\color{black} M(\varphi)\color{black}\in L^1(\Omega),
\:\:\:|\overline{\varphi}|\leq m_0\big\},
\end{align}
and $m_0\in [0,1]$ is fixed.
 The metric on
$\mathcal{X}_{m_0}$ is
\begin{align}
&\boldsymbol{d}(z_2,z_1):=
\Vert u_2-u_1\Vert+\Vert\varphi_2-\varphi_1\Vert,\nonumber
%+\Big|\int_{\Omega}F(\varphi_2)-\int_{\Omega}F(\varphi_1)\Big|^{1/2},\nonumber
\end{align}
for every $z_1:=[u_1,\varphi_1]$ and $z_2:=[u_2,\varphi_2]$ in $\mathcal{X}_{m_0}$.

In the next proposition we shall prove that the set of all weak solution is a generalized semiflow
in the sense of J.M. Ball (cf. \cite{Ba}).
% To this aim we slightly strengthen assumption (A1), that is,
%\begin{description}
%\item[(A6)] $m,F$ satisfy (A1) and there  exists $\rho\in[0,1)$
%such that
%\begin{equation}
%% \rho F_1''(s)+F_2''(s)+a(x)\geq 0,\qquad\forall s\in(-1,1),\quad\mbox{a.e. in }\Omega.\nonumber
%\end{equation}%
%\color{black}
%\end{description}
%\begin{description}
%\item[(A6)] $m,F$ satisfy (A1) and there exists $\alpha_0>0$ such that
%\begin{align}
%&m(s)F_1''(s)\geq\alpha_0,\qquad\forall s\in[-1,1].\nonumber\\
%& F_2''(s)+a(x)\geq 0,\qquad\forall s\in[-1,1],\quad\mbox{a.e. in }\Omega.\nonumber
%\end{align}
%\end{description}
%Then we prove the following
\begin{prop}
\label{gensemifl} Let $d=2$ and suppose that
the conditions of Theorem \ref{Theor2} are satisfied.
Assume also that
\eqref{cond-1}, \eqref{cond0} and that (A5) {\color{black} hold.}
% (A2)--(A6)
%and (H2) hold.
Then $\mathcal{G}_{m_0}$
is a generalized semiflow on $\mathcal{X}_{m_0}$.
\end{prop}
\begin{proof}
It is immediate to check that hypotheses (H1), (H2) and (H3) of
the definition of generalized semiflow \cite[Definition 2.1]{Ba} are
satisfied. It remains to prove the upper semicontinuity with respect to
initial data, i.e., that $\mathcal{G}_{m_0}$ satisfies (H4) of \cite[Definition
2.1]{Ba}. Take then $z_j=[u_j,\phij]\in\mathcal{G}_{m_0}$ such that $z_j(0)\to z_0$
in $\mathcal{X}_{m_0}$. Our aim is to prove that there
exist $z\in\mathcal{G}_{m_0}$ with $z(0)=z_0$ and a subsequence $\{z_{j_k}\}$
such that $z_{j_k}(t)\to z(t)$ in $\mathcal{X}_{m_0}$ for all $t\geq 0$.
Now, each weak solution $z_j=[u_j,\phij]$ satisfies the regularity
properties \eqref{reg1}--\eqref{reg6}, \eqref{phiLq} and the energy equation
\eqref{energeq} which,
 on account of \eqref{cond-1}, \eqref{cond0}
 and of the fact that $|\phij|\leq 1$, implies
\begin{align}
&
\frac{d}{dt}
\Big(\Vert u_j\Vert^2+\Vert\phij\Vert^2\Big)+(1-\rho)\alpha_0\Vert\nabla\phij\Vert^2
%2\int_\Omega \big(a+F_2''(\varphi_j)\big) m(\phij)|\nabla\phij|^2
+\nu\Vert\nabla u_j\Vert^2
\leq c+c\Vert u_j\Vert^2+\frac{1}{2\nu}\Vert h\Vert_{V_{div}'}^2,\nonumber
\end{align}
where the positive constant $c$ depends on $J$ and on $m$. By integrating this inequality between $0$
and $t$, using the fact that $z_j(0)\to z_0$ in $\mathcal{X}_{m_0}$ and Gronwall's lemma, from
the above differential inequality we get
\begin{align}
&\|u_j\|_{L^{\infty}(0,T;G_{div})\cap L^2(0,T;V_{div})}\leq C,\qquad
\|\phij\|_{L^{\infty}(0,T;H)\cap L^2(0,T;V)}\leq C.\nonumber
\end{align}
By comparison in the variational formulation \eqref{wform1} and \eqref{wform2}
written for each weak solution $z_j=[u_j,\varphi_j]$ we also obtain the estimates
for the time derivatives $u_j'$ and $\varphi_j'$
\begin{align}
&\Vert u_j'\Vert_{L^2(0,T;V_{div}')}\leq C,\qquad\Vert\phij'\Vert_{L^2(0,T;V')}\leq C.\nonumber
\end{align}
Therefore, by standard compactness results, we deduce that there exist two functions $u$ and $\varphi$ such that,
for a not relabeled subsequence, we have
\begin{align}
& u_j\rightharpoonup u\quad\mbox{weakly}^{\ast}\mbox{ in } L^{\infty}(0,T;G_{div}),
\quad\mbox{weakly in }L^2(0,T;V_{div}),\label{convj1}\\
& u_j\to u\quad\mbox{strongly in }L^2(0,T;G_{div}),\quad\mbox{a.e. in }Q_T,\label{convj2}\\
& u_j'\rightharpoonup u_t
\quad\mbox{weakly in }L^2(0,T;V_{div}'),\qquad\:\:\label{convj3}\\
& \phij\rightharpoonup\varphi\quad\mbox{weakly}^{\ast}\mbox{ in }L^{\infty}(0,T;H),
\quad\mbox{weakly in }L^2(0,T;V),\label{convj4}\\
& \phij\to\varphi\quad\mbox{strongly in }L^2(0,T;H),\quad\mbox{a.e. in }Q_T,\label{convj5}\\
& \phij'\to\varphi_t\quad\mbox{weakly in }L^2(0,T;V'). \label{convj6}
\end{align}
Passing to the limit in the variational formulation for $z_j=[u_j,\varphi_j]$ and using
the weak/strong convergences \eqref{convj1}--\eqref{convj6}
and the fact that $|\phij|\leq 1$, we
get that $z:=[u,\varphi]\in\mathcal{G}_{m_0}$ (the argument is similar to the passage to the limit in
\eqref{wsappb1} and \eqref{wsappb2} in the proof of Theorem \ref{Theor2}).
Furthermore we have $z(0)=z_0$, since, as a consequence of
\eqref{convj1}, \eqref{convj3}, \eqref{convj4} and \eqref{convj6}, we have, for all $t\geq 0$
\begin{align}
& u_j(t)\rightharpoonup u(t)\qquad\mbox{weakly in }\:\:G_{div},\label{wcon1}\\
& \phij(t)\rightharpoonup \varphi(t)\qquad\mbox{weakly in }\:\:H.\label{wcon2}
\end{align}
Let us now see that $z_j(t)\to z(t)$ in $\mathcal{X}_{m_0}$ for all $t\geq 0$. First observe that
the energy equation \eqref{energeq} can be written in the form
\begin{align}
&\frac{d}{dt}\widetilde{E}\big(z_j(t)\big)+\int_\Omega m(\phij)F''(\phij)|\nabla\phij|^2
+\int_\Omega a m(\phij)|\nabla\phij|^2+\nu\Vert\nabla u_j\Vert^2=0,
\label{enj}
\end{align}
where
\begin{align}
&\widetilde{E}\big(z_j(t)\big):=E\big(z_j(t)\big)-\int_0^t\int_\Omega m(\phij)\big(\nabla J\ast\phij-\phij\nabla a)\cdot\nabla\phij\nonumber\\
&-\int_0^t\int_\Omega(a\phij-J\ast\phij) u_j\cdot\nabla\phij-\int_0^t\langle h,u_j\rangle d\tau,\nonumber
\end{align}
and
\begin{align}
& E\big(z_j(t)\big):=\frac{1}{2}\big(\Vert u_j(t)\Vert^2+\Vert\phij(t)\Vert^2\big).\nonumber
\end{align}
On the other hand, from \eqref{convj2} and \eqref{convj5} we have, for almost any $t>0$,
\begin{align}
& u_j(t)\to u(t)\qquad\mbox{strongly in }\:\:G_{div},\nonumber\\
& \phij(t)\to \varphi(t)\qquad\mbox{strongly in }\:\:H,\nonumber
\end{align}
and hence $E\big(z_j(t)\big)\to E\big(z(t)\big)$ for almost any $t>0$.
Moreover, due to the condition $|\varphi_j|\leq 1$ for each $j$, to the pointwise convergence
\eqref{convj5} and Lebesgue's theorem we have
\begin{align}
&\varphi_j\to\varphi,\qquad\mbox{ strongly in }\:\: L^s(Q_T),\qquad\forall s\in[2,\infty).\nonumber
\end{align}
%it is easy to see that \eqref{convj4} implies the strong convergence \eqref{conve13}
%(written for the sequence $\{\varphi_j\}$).
Therefore, recalling
that $m(\varphi_j)\to m(\varphi)$ strongly in $L^s(Q_T)$ for all $s\in [2,\infty)$ and using
the second weak convergence in \eqref{convj4}, we have, for all $t\geq 0$,
\begin{align}
&\int_0^t\int_\Omega m(\phij)\big(\nabla J\ast\phij-\phij\nabla a)\cdot\nabla\phij
\to\int_0^t\int_\Omega m(\varphi)\big(\nabla J\ast\varphi-\varphi\nabla a)\cdot\nabla\varphi.\nonumber
\end{align}
Furthermore, since $a\phij-J\ast\phij\to a\varphi-J\ast\varphi$ strongly
in $L^s(Q_T)$ for all $s\in[2,\infty)$ and, by \eqref{convj1},
$u_j\to u$ strongly in $L^s(Q_T)^2$ for all $s\in[2,4)$, then we also have, for all $t\geq 0$
\begin{align}
&\int_0^t\int_\Omega(a\phij-J\ast\phij) u_j\cdot\nabla\phij\to
\int_0^t\int_\Omega(a\varphi-J\ast\varphi) u\cdot\nabla\varphi.\nonumber
\end{align}
Therefore
\begin{align}
&\widetilde{E}\big(z_j(t)\big)\to\widetilde{E}\big(z(t)\big),\qquad\mbox{ a.a. }\:\:t>0.\nonumber
\end{align}
Since, due to \eqref{enj}, $\widetilde{E}\big(z_j(\cdot)\big)$ is non increasing in $[0,\infty)$ for every $j$
and $\widetilde{E}\big(z_j(\cdot)\big),\widetilde{E}\big(z(\cdot)\big)\in C([0,\infty))$,
we infer
\begin{align}
&\widetilde{E}\big(z_j(t)\big)\to\widetilde{E}\big(z(t)\big),\qquad\forall\:\:t\geq 0,\nonumber
\end{align}
and hence also $ E\big(z_j(t)\big)\to E\big(z(t)\big)$ for all $t\geq 0$. This last convergence
together with \eqref{wcon1} and \eqref{wcon2} yield
$u_j(t)\to u(t)$ strongly in $G_{div}$ and
$\phij(t)\to \varphi(t)$ strongly in $H$, and hence $z_j(t)\to z(t)$ in $\mathcal{X}_{m_0}$, for all $t\geq 0$.
\end{proof}

\begin{prop}
\label{dissest} Let $d=2$
and suppose that
the conditions of Theorem \ref{Theor2} are satisfied.
Assume also that \eqref{cond-1}, \eqref{cond0} and (A5) hold.
\color{black}
%suppose that (A2)-(A6) and (H2) hold.
%Furthermore, assume that the following condition is satisfied
%\begin{align}
%&\nu>\frac{1}{\alpha\lambda_1}\Vert J\Vert_{L^1}^2.
%\label{viscond}
%\end{align}
Then $\mathcal{G}_{m_0}$ is point dissipative and eventually bounded.
\end{prop}
\begin{proof}

Let us estimate the second term on the right hand side of \eqref{energeq} as follows
\begin{align}
&\Big|\int_\Omega(a\varphi-J\ast\varphi)u\cdot\nabla\varphi\Big|
=\Big|\int_{\Omega}\Big(\frac{1}{2}\varphi^2\nabla a-\varphi(\nabla J\ast\varphi)\Big)\cdot u\Big|\nonumber\\
&\leq\frac{3}{2}b|\Omega|^{1/2}\Vert u\Vert\leq\frac{\nu}{4}\Vert\nabla u\Vert^2+C_0,\nonumber
\end{align}
where $C_0=9b^2|\Omega|/4\nu\lambda_1$ with the constant $b$ defined as in (H2).
% and $N=\sup_{x\in\Omega}\int_{\Omega}|\nabla J(x-y)|dy$.
Moreover, the first term on the right hand side of \eqref{energeq} can be controlled in the following way
\begin{align}
&\Big|\int_\Omega m(\varphi)(\nabla J\ast\varphi-\varphi\nabla a)\cdot\nabla\varphi\Big|
\leq 2m_0|\Omega|^{1/2}b\Vert\nabla\varphi\Vert\leq(1-\rho)
\frac{\alpha_0}{2}\Vert\nabla\varphi\Vert^2+C_1,\nonumber
\end{align}
where $C_1=2 m_\ast^2|\Omega|b^2/\alpha_0$, and $m_\ast=\max_{s\in[-1,1]}m(s)$. Then, by taking
 \eqref{cond-1} and \eqref{cond0}
into account, we get the differential inequality
\begin{align}
&\frac{d}{dt}\big(\Vert u\Vert^2+\Vert\varphi\Vert^2\big)+(1-\rho)
\alpha_0\Vert\nabla\varphi\Vert^2+\nu\Vert\nabla u\Vert^2
\leq C_2+\frac{1}{\nu}\Vert h\Vert_{V_{div}'}^2,\nonumber
\end{align}
where $C_2=2(C_0+C_1)$.
%Let us estimate the second term on the right hand side of \eqref{energeq} as
%\begin{align}
%&\Big|\int_\Omega(a\varphi-J\ast\varphi)u\cdot\nabla\varphi\Big|\leq 2\Vert J\Vert_{L^1}
%\Vert u\Vert\Vert\nabla\varphi\Vert\nonumber\\
%& \leq(\alpha-\e)\Vert\nabla\varphi\Vert^2+\frac{1}{\lambda_1(\alpha-\e)}\Vert J\Vert_{L^1}^2\Vert\nabla u\Vert^2,
%\nonumber
%\end{align}
%where $0<\e<\alpha$ is to be fixed later. Then, \eqref{energeq} leads to
%\begin{align}
%&
%\frac{1}{2}\frac{d}{dt}
%\big(\Vert u\Vert^2+\Vert\varphi\Vert^2\big)+\frac{\e}{2}\Vert\nabla\varphi\Vert^2
%+\int_\Omega a m(\varphi)|\nabla\varphi|^2
%+\Big(\nu-\e'-\frac{\Vert J\Vert_{L^1}^2}{\lambda_1(\alpha-\e)}\Big)\Vert\nabla u\Vert^2\nonumber\\
%&\leq\frac{c_m^2}{2\e}\Vert\nabla J\Vert_{L^1}^2+\frac{1}{4\e'}\Vert h\Vert_{V_{div}'}^2,
%\label{est13}
%\end{align}
 %where $\e'>0$.
%Let us now choose $\e$ and $\e'$ small enough, i.e., such that
%$$\nu-\e'-\frac{\Vert J\Vert_{L^1}^2}{\lambda_1(\alpha-\e)}\geq\frac{\nu_\ast}{2},\qquad
%\nu_\ast:=\nu-\frac{\Vert J\Vert_{L^1}^2}{\alpha\lambda_1}.$$
%This is possible thanks to condition \eqref{viscond}.
%Therefore, from \eqref{est13} we get
%\begin{align}
%&\frac{d}{dt}
%\big(\Vert u\Vert^2+\Vert\varphi\Vert^2\big)+\e\Vert\nabla\varphi\Vert^2
%+\nu_\ast\Vert\nabla u\Vert^2\leq
%\frac{c_m^2}{\e}\Vert\nabla J\Vert_{L^1}^2+\frac{1}{2\e'}\Vert h\Vert_{V_{div}'}^2.\nonumber
%\end{align}
By using the identity
 $\Vert\varphi\Vert^2=\Vert\varphi-\overline{\varphi}_0\Vert^2+\Vert\overline{\varphi}_0\Vert^2$,
 we obtain
\begin{align}
&\frac{d}{dt}
\big(\Vert u\Vert^2+\Vert\varphi-\overline{\varphi}_0\Vert^2\big)+(1-\rho)
\alpha_0 C_P\Vert\varphi-\overline{\varphi}_0\Vert^2
+\nu\lambda_1\Vert u\Vert^2\leq C_2+\frac{1}{\nu}\Vert h\Vert_{V_{div}'}^2,\nonumber
\end{align}
where
%$C_1$ depends on $m$, $J$ and $\Vert h\Vert_{V_{div}'}$ and
$C_P$
is the constant appearing in the Poincar\'{e}-Wirtinger inequality. Therefore we get
\begin{align}
&\frac{d}{dt} E\big(\widetilde{z}(t)\big)+\eta E\big(\widetilde{z}(t)\big)\leq C_3,
\label{est14}
\end{align}
where $\widetilde{z}=[u,\varphi-\overline{\varphi}_0]$ and the constants $\eta$ and $C_3$
are given by $\eta=\min((1-\rho)\alpha_0 C_P,\nu\lambda_1)/2$, and
$C_3=C_2/2+\Vert h\Vert_{V_{div}'}^2/2\nu$, respectively.
By Gronwall's lemma from \eqref{est14} we have
$$ E\big(\widetilde{z}(t)\big)\leq E\big(\widetilde{z}(0)\big)e^{-\eta t}+\frac{C_3}{\eta},\qquad\forall t\geq 0.$$
This estimate easily yields
\begin{align}
&\boldsymbol{d}^2(z(t),0)\leq\boldsymbol{d}^2(z_0,0)e^{-\eta t}
+\frac{2 C_3}{\eta}+|\overline{\varphi}_0|^2|\Omega|,\qquad\forall t\geq 0,\nonumber
\end{align}
which entails both the point dissipativity and the eventual boundedness of $\mathcal{G}_{m_0}$.
\end{proof}

We can now state the result on the existence of the global attractor.
\begin{prop}
Let the assumptions of Proposition \ref{dissest} hold.
% and suppose that that (A2)-(A6) and (H2) hold.
%Furthermore, assume that \eqref{viscond} is satisfied.
Then $\mathcal{G}_{m_0}$ possesses a global attractor.
\end{prop}
\begin{proof}\label{attractor}
Thanks to Proposition \ref{dissest} and by \cite[Proposition 3.2]{Ba}
and \cite[Theorem 3.3]{Ba}
we only need to show that $\mathcal{G}_{m_0}$ is compact.
Let $\{z_j\}\subset\mathcal{G}$ be a sequence with $\{z_j(0)\}$ bounded in $\mathcal{X}_{m_0}$. We claim that
there exists a subsequence $\{z_{j_k}\}$ such that $z_{j_k}(t)$ converges in $\mathcal{X}_{m_0}$
for every $t>0$.
Indeed, the energy equation \eqref{energeq} entails the existence of a subsequence (not relabeled)
such that (cf. the proof of Proposition \ref{gensemifl}), for almost all $t>0$,
$$u_j(t)\to u(t)\quad\mbox{strongly in }G_{div},\qquad
\varphi_j(t)\to\varphi(t)\quad\mbox{strongly in }H\mbox{ and a.e. in }\Omega,$$
where $z=[u,\varphi]$ is a weak solution.
By arguing as in the proof of Proposition \ref{gensemifl} we infer that
$\widetilde{E}\big(z_j(t)\big)\to\widetilde{E}\big(z(t)\big)$ for all $t\geq 0$.
Thus $z_j(t)\to z(t)$ in $\mathcal{X}_{m_0}$ for $t>0$,
which yields the compactness of $\mathcal{G}_{m_0}$.
\end{proof}

\begin{oss}
\label{purephases}
{\upshape
We point out that the existence of the global attractor has been
established without the restriction $|\overline{\varphi}|<1$ on
the generalized semiflow  (compare with Remark \ref{pure}).
In particular, this result does not
require the separation property.
It is also worth observing that in dimension three the trajectory attractor approach used in \cite{FG1,FG2}
might be extended to the present case.}
\end{oss}

{\color{black}
\begin{oss}
\label{uniqueness}
As mentioned in the Introduction, uniqueness of weak solutions in two dimensions has been recently proved in \cite{FGG}. As a consequence,
the generalized semiflow becomes a semigroup and the global attractor is connected (see \cite[Section 5]{FGG}).
\end{oss}}}

\section{The convective nonlocal Cahn-Hilliard equation\\ with degenerate mobility}\setcounter{equation}{0}
\label{sec:convCahnHill}

By relying on the results of the previous sections we can prove
similar results for the convective nonlocal Cahn-Hilliard equation
with degenerate mobility and with a given velocity field, for
$d=2,3$. In particular, from Theorem \ref{Theor2} it is
straightforward to deduce the following

\begin{thm}
\label{nlocCH1} Assume that (A1)-(A4) and (H2) hold. Let
$u\in L^2_{loc}([0,\infty);V_{div}\cap L^\infty(\Omega)^d)$ be
given and let $\varphi_0\in L^{\infty}(\Omega)$ be such that
$F(\varphi_0)\in L^1(\Omega)$ and $M(\varphi_0)\in L^1(\Omega)$.
Then, for every $T>0$, there exists a weak solution $\varphi$ to
\eqref{sy1}, \eqref{sy2}, \eqref{sy5}$_1$ on $[0,T]$ corresponding
to $\varphi_0$ fulfilling \eqref{reg4}--\eqref{reg6}, the weak
formulation \eqref{wform1}, and such that
$\overline{\varphi}(t)=\overline{\varphi_0}$ for all $t\in[0,T]$.
Furthermore, $\varphi$ satisfies
\begin{align}
&\varphi\in L^{\infty}(0,T;L^p(\Omega)),
\end{align}
where $p\leq 6$ for $d=3$ and $2\leq p<\infty$ for $d=2$. In
addition, for almost any $t>0$, the following energy identity holds
\begin{align}
& \frac{1}{2}\frac{d}{dt} \Vert\varphi\Vert^2+\int_\Omega
m(\varphi)F''(\varphi)|\nabla\varphi|^2 +\int_\Omega a
m(\varphi)|\nabla\varphi|^2 +\int_\Omega
m(\varphi)\big(\varphi\nabla a-\nabla
J\ast\varphi\big)\cdot\nabla\varphi=0.\label{nlocCH7}
\end{align}

\end{thm}

\begin{oss}{\upshape
The basic estimates in the proof of Theorem \ref{nlocCH1}
are now obtained from \eqref{nlocCH7} (written for the approximate solution
$\phie$ and integrated between $0$ and $t$), in place of \eqref{eiappeps}. In particular,
\eqref{nlocCH7} still yields \eqref{estep2}, \eqref{estep4} and \eqref{estep5},
so that we can argue as in the proof of Theorem \ref{Theor3}.}
\end{oss}

For the convective nonlocal Cahn-Hilliard equation with degenerate mobility we can also prove a uniqueness result (see \cite{GL2} for more restrictive assumptions).
We remind that uniqueness of solutions is an open issue for the local case (see \cite{EG}) as well as for the complete
system \eqref{sy1}-\eqref{sy6} even in dimension two.

\begin{prop}\label{nlocCH6}
Let
% \eqref{cond0}, \eqref{ass}, and
all the assumptions of Theorem \ref{nlocCH1} and \eqref{cond-1}, \eqref{cond0} be satisfied.
Then, the weak solution to \eqref{sy1}, \eqref{sy2}, \eqref{sy5}$_1$, \eqref{sy6}$_2$ is unique.
\end{prop}
\begin{proof}
Following \cite{GL2}, let us introduce
\begin{align}
&\Lambda_1(s):=\int_0^s m(\sigma)F_1''(\sigma)d\sigma,\qquad\Lambda_2(s):=\int_0^s m(\sigma)F_2''(\sigma)d\sigma,\qquad\Gamma(s):=\int_0^s m(\sigma)d\sigma,
\nonumber
\end{align}
for all $s\in[-1,1]$.
Due to (A1) and \eqref{cond0} we have $\Lambda_1\in C^1([-1,1])$ and
$0<\alpha_0\leq \Lambda_1'(s)\leq a$ for some positive constant $a$.
Then, it is easy to see that the weak formulation \eqref{wform1} can be rewritten as follows
\begin{align}
&\langle\varphi_t,\psi\rangle+\big(\nabla\Lambda(\cdot,\varphi),\nabla\psi\big)
-\big(\Gamma(\varphi)\nabla a,\nabla\psi\big)
+\big( m(\varphi)(\varphi\nabla a-\nabla J\ast\varphi),\nabla\psi\big)
=\big(u\varphi,\nabla\psi\big),
%\qquad\forall\psi\in V.
\label{nlocCH2}
\end{align}
for all $\psi\in V$, where
$$\Lambda(x,s):=\Lambda_1(s)+\Lambda_2(s)+a(x)\Gamma(s).$$
Consider now \eqref{nlocCH2} for two weak solutions $\varphi_1$ and $\varphi_2$
\color{black} corresponding to the same initial datum. \color{black}
Let us take the difference between the two identities,
set $\varphi:=\varphi_1-\varphi_2$ and
choose $\psi=\mathcal{N}\varphi$ as test function
in the resulting identity (notice that $\overline{\varphi}=0$). This yields
\begin{align}
&\frac{1}{2}\frac{d}{dt}\Vert\mathcal{N}^{1/2}\varphi\Vert^2
+\big(\Lambda(\varphi_2)-\Lambda(\varphi_1),\varphi\big)
-\big((\Gamma(\varphi_2)-\Gamma(\varphi_1))\nabla a,\nabla\mathcal{N}\varphi\big)\nonumber\\
&+\big((m(\varphi_2)-m(\varphi_1))(\varphi_2\nabla a-\nabla J\ast\varphi_2)+m(\varphi_1)(\varphi\nabla a
-\nabla J\ast\varphi),\nabla\mathcal{N}\varphi\big)\nonumber\\
&=\big(u\varphi,\nabla\mathcal{N}\varphi\big).\label{nlocCH3}
\end{align}
On account of  \eqref{cond-1} and \eqref{cond0}, we find
\begin{align}
&\big(\Lambda(\cdot,\varphi_2)-\Lambda(\cdot,\varphi_1),\varphi\big)\geq
(1-\rho)\int_\Omega m(\theta\varphi_2+(1-\theta)\varphi_1)F_1''(\theta\varphi_2+(1-\theta)\varphi_1)
\varphi^2\nonumber\\
&\geq (1-\rho)\alpha_0\Vert\varphi\Vert^2.\label{nolcCH4}
\end{align}
Furthermore, since $|\varphi_1|,|\varphi_2|\leq 1$, then we have
\begin{align}
&\big|\big((\Gamma(\varphi_2)-\Gamma(\varphi_1))\nabla a,\nabla\mathcal{N}\varphi\big)\big|
%\Vert\Gamma(\varphi_2)-\Gamma(\varphi_1)\Vert\Vert\nabla a\Vert_{L^\infty(\Omega)}
%\Vert\nabla\mathcal{N}\varphi\Vert
%\nonumber\\
%&
\leq m_\ast b \Vert\varphi\Vert\Vert\nabla\mathcal{N}\varphi\Vert\leq
\frac{1}{8} (1-\rho)\alpha_0\Vert\varphi\Vert^2+C_1\Vert\nabla\mathcal{N}\varphi\Vert^2,\\
&\big|(m(\varphi_2)-m(\varphi_1))(\varphi_2\nabla a-\nabla J\ast\varphi_2),\nabla\mathcal{N}\varphi\big)\big|
\leq 2m_{\ast\ast}b\Vert\varphi\Vert
\Vert\nabla\mathcal{N}\varphi\Vert\nonumber\\
&\leq\frac{1}{8} (1-\rho)\alpha_0\Vert\varphi\Vert^2+C_2\Vert\nabla\mathcal{N}\varphi\Vert^2,\\
&\big|\big(m(\varphi_1)(\varphi\nabla a
-\nabla J\ast\varphi),\nabla\mathcal{N}\varphi\big)\big|\leq 2m_\ast b\Vert\varphi\Vert\Vert\nabla\mathcal{N}\varphi\Vert\nonumber\\
&\leq \frac{1}{8} (1-\rho)\alpha_0\Vert\varphi\Vert^2+C_3\Vert\nabla\mathcal{N}\varphi\Vert^2,                                                              \\
%&\big(J\ast\varphi,\varphi\big)\leq \Vert\mathcal{N}^{-1/2}(J\ast\varphi)\Vert\Vert\mathcal{N}^{1/2}\varphi\Vert
%\leq\frac{\alpha_0}{4}\Vert\varphi\Vert^2+C_1\Vert\mathcal{N}^{1/2}\varphi\Vert^2,\\
&\big|\big(u\varphi,\nabla\mathcal{N}\varphi\big)\big|
\leq \Vert u\Vert_{L^\infty(\Omega)^d}\Vert\varphi\Vert\Vert\mathcal{N}^{1/2}\varphi\Vert\nonumber\\
&\leq\frac{1}{8} (1-\rho)\alpha_0\Vert\varphi\Vert^2+C_4\Vert u\Vert_{L^\infty(\Omega)^d}^2\Vert\mathcal{N}^{1/2}\varphi\Vert^2,
\label{nlocCH5}
\end{align}
where $b$ is given as in (H2), $m_\ast$  as in the proof of Proposition \ref{dissest},
$m_{\ast\ast}:=\max_{s\in[-1,1]} m'(s)$ and the positive constants $C_1,..,C_4$
depend on $\alpha_0,\rho,m_\ast,m_{\ast\ast}$.
%The conclusion follows
By plugging \eqref{nolcCH4}--\eqref{nlocCH5} into \eqref{nlocCH3} we get
\begin{align}
&\frac{d}{dt}\Vert\mathcal{N}^{1/2}\varphi\Vert^2+(1-\rho)\alpha_0\Vert\varphi\Vert^2
\leq C_5\big(1+\Vert u\Vert_{L^\infty(\Omega)^d}^2\big)\Vert\mathcal{N}^{1/2}\varphi\Vert^2
\end{align}
and Gronwall's lemma applies.
\end{proof}
\color{black}

As a consequence of Theorem \ref{nlocCH1} and of Proposition \ref{nlocCH6}
we can define a semiflow $S(t)$ on $\mathcal{Y}_{m_0}$ (cf. \eqref{phasesp2}), $m_0\in[0,1]$,
endowed with the metric induced by the $L^2-$norm. It is then immediate to check that the arguments used in the proofs
of Proposition \ref{gensemifl}, Proposition \ref{dissest}
and Proposition \ref{attractor} can be adapted to the present situation.
This gives
\begin{thm}
\label{attrnlch}
Let (A1)-(A4), \eqref{cond-1} and \eqref{cond0} hold.
Suppose that  $u\in L^\infty(\Omega)^d\color{black}\cap V_{div} \color{black}$ is given and independent of time.
Then the dynamical system $\big(\mathcal{Y}_{m_0},S(t)\big)$
possesses a connected global attractor.
\end{thm}

\begin{oss}
{\upshape The fact $u$ must be divergence-free is not necessary (see also \cite{FG1,FG2}).
In particular, the convective term can be of the form $\nabla\cdot(u \varphi)$.}
\end{oss}

\begin{oss}
{\upshape {\color{black} In \cite[Section 4]{GG4} the case $u=0$ was considered. In particular,
the existence of an exponential attractor was proven and, as a consequence,
the existence of a global attractor of finite fractal dimension. However, contrary to the present
case, in \cite{GG4} pure phases were a priori excluded from the phase space (cf. also Remark \ref{purephases}).}}
\end{oss}

\color{black}
\textbf{Acknowledgments}. The authors would like to thank Ciprian G. Gal for a useful remark concerning
the first part of the proof of Theorem \ref{Theor3}. The authors are members of the Gruppo Nazionale per l'Analisi Matematica, la Probabilit\`{a} e le loro Applicazioni (GNAMPA) of the Istituto Nazionale di Alta Matematica (INdAM).
\color{black}


\begin{thebibliography}{30}

\bibitem{A1} H. Abels, {\itshape On a diffusive interface model for
    two-phase flows of viscous, incompressible fluids with matched
    densities}, Arch. Ration. Mech. Anal.  \textbf{194} (2009),
    463-506.

\bibitem{A2} H. Abels, {\itshape Longtime behavior of
    solutions of a Navier-Stokes/Cahn-Hilliard system}, Proceedings of the
    Conference ``Nonlocal and Abstract Parabolic Equations and their
    Applications'', Bedlewo, Banach Center Publ. \textbf{86} (2009), 9-19.

\bibitem{ADG1} H. Abels, D. Depner, H. Garcke, {\itshape
Existence of weak solutions for a diffuse interface model for two-phase
flows of incompressible fluids with different densities}, J. Math. Fluid Mech. \textbf{15} (2013), 453-480.
%DOI 10.1007/s00021-012-0118-x.

\bibitem{ADG2} H. Abels, D. Depner, H. Garcke,
{\itshape On an incompressible Navier-Stokes/Cahn-Hilliard system
with degenerate mobility}, Ann. Inst. H. Poincar\'e
Anal. Non Lin\'eaire \textbf{30} (2013), 1175-1190.

{\color{black}
\bibitem{ADT} H. Abels, L. Diening, Y. Terasawa, {\itshape Existence of Weak Solutions for a Diffuse Interface Model of Non-Newtonian Two-Phase Flows}, Nonlinear Anal. Real World Appl. \textbf{15} (2014), 149-157.}

\bibitem{AbFei} H. Abels, E. Feireisl, {\itshape On a diffuse interface model for a two-phase flow of
compressible viscous fluids}, Indiana Univ. Math. J. \textbf{57} (2008), 659--698.

\bibitem{AMW} D.M. Anderson, G.B. McFadden, A.A. Wheeler,
    {\itshape Diffuse-interface methods in fluid mechanics},
    Annu. Rev. Fluid Mech. \textbf{30}, Annual Reviews, Palo
    Alto, CA, 1998, 139-165.

\bibitem{Ba} J.M. Ball, {\itshape Continuity properties and global
    attractors of generalized semiflows and the Navier-Stokes
    equation}, J. Nonlinear Sci. \textbf{7} (1997), 475-502 (Erratum, J.
    Nonlinear Sci. \textbf{8} (1998), 233).

{\color{black}
\bibitem{BB} J.W. Barrett, J.F. Blowey, {\itshape Finite element approximation of the Cahn-Hilliard equation with concentration dependent mobility},  Math. Comp. \textbf{68}  (1999), 487-517.}

\bibitem{BH1} P.W. Bates, J. Han, {\itshape The Neumann
    boundary problem for a nonlocal Cahn-Hilliard equation}, J.
    Differential Equations \textbf{212} (2005), 235-277.

\bibitem{Bos} S. Bosia, {\itshape Analysis of a Cahn-Hilliard-Ladyzhenskaya system with singular potential},
J. Math. Anal. Appl. \textbf{397} (2013), 307-321.

\bibitem{B} F. Boyer, {\itshape Mathematical study of multi-phase
    flow under shear through order parameter formulation}, Asymptot.
    Anal. \textbf{20} (1999), 175-212.

{\color{black}

\bibitem{CG} C. Cao, C.G. Gal, {\itshape Global solutions for the 2D NS-CH model for a
two-phase flow of viscous, incompressible fluids with mixed partial
viscosity and mobility}, Nonlinearity \textbf{25} (2012), 3211-3234.}

%\bibitem{CMZ} L. Cherfils, A. Miranville, S. Zelik, The Cahn-Hilliard equation with logarithmic potentials, Milan J. Math.,
%\textbf{79} (2011), 561-596.}

%\bibitem{C} J.W. Cahn, {\itshape On spinodal decomposition}, Acta Metall.
%    \textbf{9} (1961), 795-801.

%\bibitem{CH} J.W. Cahn, J.E. Hilliard, {\itshape Free energy of a
%    nonuniform system. I. Interfacial free energy}, J. Chem. Phys.
%    \textbf{28} (1958), 258-267.

%\bibitem{CH2} J.W. Cahn, J.E. Hilliard, {\itshape Spinodal decomposition: A reprise},
%    Acta Metall. \textbf{19} (1971), 151-161.


\bibitem{CFG} P. Colli, S. Frigeri, M. Grasselli, {\itshape Global
    existence of weak solutions to a nonlocal
    Cahn-Hilliard-Navier-Stokes system}, J. Math. Anal. Appl.
    \textbf{386} (2012), 428-444.

\bibitem{CKRS} P. Colli, P. Krej\v{c}\'{i}, E. Rocca, J. Sprekels,
    {\itshape Nonlinear evolution inclusions arising from phase change
    models}, Czechoslovak Math. J. \textbf{57} (2007), 1067-1098.

\bibitem{EG} C.M. Elliott, H. Garcke, {\itshape On the Cahn-Hilliard equation with degenerate mobility},
     SIAM J. Math. Anal. \textbf{27} (1996), 404-423.

\bibitem{Em} H. Emmerich, The diffuse interface approach in materials science,
Springer, Berlin Heidelberg, 2003.

{\color{black}
\bibitem{FGG} S. Frigeri, C.G. Gal, M. Grasselli, {\itshape On nonlocal Cahn-Hilliard-Navier-Stokes systems in two dimensions}, submitted.} {\color{black} Wias Preprint \textbf{1923}, (2014), 34 pp.}

\bibitem{FG1} S. Frigeri, M. Grasselli, {\itshape Global and trajectories attractors for a nonlocal
Cahn-Hilliard-Navier-Stokes system}, J. Dynam. Differential Equations \textbf{24} (2012), 827-856.

\bibitem{FG2} S. Frigeri, M. Grasselli, {\itshape Nonlocal Cahn-Hilliard-Navier-Stokes systems
with singular potentials}, Dyn. Partial Differ. Equ. \textbf{9} (2012), 273-304.

\bibitem{FGK} S. Frigeri, M. Grasselli, P. Krej\v{c}\'{\i}, {\itshape Strong solutions for
two-dimensional nonlocal Cahn-Hilliard-Navier-Stokes systems,}
{\color{black} J. Differential Equations \textbf{255} (2013), 2587-2614.}

\bibitem{GG1} C.G. Gal, M. Grasselli, {\itshape Asymptotic
    behavior of a Cahn-Hilliard-Navier-Stokes system in 2D},  Ann. Inst. H. Poincar\'e
Anal. Non Lin\'eaire \textbf{27} (2010), 401-436.

\bibitem{GG2} C.G. Gal, M. Grasselli, {\itshape Trajectory attractors
    for binary fluid mixtures in 3D}, Chinese Ann. Math. Ser. B
    \textbf{31} (2010), 655-678.

\bibitem{GG3} C.G. Gal, M. Grasselli, {\itshape
Instability of two-phase flows: a lower bound on the dimension of the
global attractor of the Cahn-Hilliard-Navier-Stokes system,} Phys. D \textbf{240} (2011), 629-635.

\bibitem{GG4} C.G. Gal, M. Grasselli, {\itshape Longtime behavior of nonlocal Cahn-Hilliard equations},
Discrete Contin. Dyn. Syst. Ser. A   {\color{black} \textbf{34} (2014), 145-179.}

\bibitem{GZ} H. Gajewski, K. Zacharias, {\itshape On a nonlocal phase separation model},
J. Math. Anal. Appl. \textbf{286} (2003), 11-31.

\bibitem{GL1} G. Giacomin, J.L. Lebowitz, {\itshape Phase
    segregation dynamics in particle systems with long range
    interactions. I. Macroscopic limits}, J. Statist. Phys. \textbf{87}
    (1997), 37-61.

\bibitem{GL2} G. Giacomin, J.L. Lebowitz, {\itshape Phase
    segregation dynamics in particle systems with long range
    interactions. II. Phase motion}, SIAM J. Appl. Math.  \textbf{58}
    (1998), 1707-1729.

\bibitem{GLM} G. Giacomin, J.L. Lebowitz, R. Marra, {\itshape Macroscopic evolution of particle systems with short- and long-range interactions}, Nonlinearity \textbf{13} (2000), 2143-2162.

\bibitem{GP} M. Grasselli, D. Pra\v{z}\'{a}k, {\itshape Longtime behavior of a diffuse interface model for binary fluid mixtures with shear dependent viscosity}, Interfaces Free Bound. \textbf{13} (2011), 507-530.

\bibitem{GPV} M.E. Gurtin, D. Polignone, J. Vi\~{n}als,
    {\itshape Two-phase binary fluids and immiscible fluids described by
    an order parameter}, Math. Models Meth. Appl. Sci. \textbf{6} (1996),
   8-15.

\bibitem{HMR} M. Heida, J. M\'{a}lek, K.R. Rajagopal, {\itshape On
    the development and generalizations of Cahn-Hilliard
    equations within a thermodynamic framework}, Z. Angew. Math. Phys.
    \textbf{63} (2012), 145-169.

\bibitem{HH} P.C. Hohenberg, B.I. Halperin, {\itshape Theory
    of dynamical critical phenomena}, Rev. Mod. Phys.
    \textbf{49} (1977), 435-479.

\bibitem{JV} D. Jasnow, J. Vi\~{n}als, {\itshape Coarse-grained
    description of thermo-capillary flow}, Phys. Fluids \textbf{8}
    (1996), 660-669.

{\color{black}
\bibitem{Kim2012} J.S. Kim, {\itshape Phase-field models for multi-component
fluid flows}, Commun. Comput. Phys., \textbf{12} (2012), 613-661.}

\bibitem{KCR} N. Kim, L. Consiglieri, J.F. Rodrigues, {\itshape On non-Newtonian incompressible fluids
with phase transitions}, Math. Methods Appl. Sci. \textbf{29} (2006), 1523-1541.

\bibitem{KRS} P. Krej\v c\'{\i}, E. Rocca, J. Sprekels,
{\itshape A nonlocal phase-field model  with nonconstant specific
heat}, Interfaces Free Bound., \textbf{9} (2007), 285-306.

\bibitem{KRS2} P. Krej\v c\'{\i}, E. Rocca, J. Sprekels,
{\itshape Non-local temperature dependent phase-field model  for non-isothermal phase transitions},
J. London Math. Soc., \textbf{76} (2007), 197-210.

\bibitem{LMM} A.G. Lamorgese, D. Molin, R. Mauri, {\itshape Phase field approach
to multiphase flow modeling}, Milan J. Math. \textbf{79} (2011), 597-642.

{\color{black}
\bibitem{LMS} S. Lisini, D. Matthes, G. Savar\'{e}, {\itshape Cahn-Hilliard and thin film equations with nonlinear mobility as gradient flows in weighted-Wasserstein metrics}, J. Differential Equations \textbf{253} (2012),  814-850.}

\bibitem{LS} C. Liu, J. Shen, {\itshape A phase field model for the mixture of two incompressible
fluids and its approximation by a Fourier spectral method}, Phys. D \textbf{179} (2003), 211-228.

\bibitem{LP} S.-O. Londen, H. Petzeltov\'{a}, {\itshape Convergence of solutions of a
non-local phase-field system}, Discrete Contin. Dyn. Syst. Ser. S \textbf{4} (2011),
653-670.

\bibitem{LP2} S.-O. Londen, H. Petzeltov\'{a},
{\itshape Regularity and separation from potential barriers for a non-local phase-field system},
J. Math. Anal. Appl. \textbf{379} (2011), 724-735.

\bibitem{M} A. Morro, {\itshape Phase-field models of Cahn-Hilliard
    Fluids and extra fluxes}, Adv. Theor. Appl. Mech. \textbf{3} (2010),
    409-424.

\bibitem{ANC} A. Novick-Cohen,{\itshape The Cahn-Hilliard equation},
Handbook of differential equations: evolutionary equations. Vol. IV,  201-228,
Handb. Differ. Equ., Elsevier/North-Holland, Amsterdam, 2008.

\bibitem{Ro}
J.S. Rowlinson, Translation of J.D. van der Waals, {\itshape
The thermodynamic theory of capillarity under the hypothesis of a continuous variation of density},
J. Statist. Phys. \textbf{20} (1979), 197-244.

{\color{black}
\bibitem{Sc}
G. Schimperna, {\itshape Global attractors for Cahn-Hilliard equations with nonconstant mobility}, Nonlinearity \textbf{20} (2007), 2365-2387.}

\bibitem{SZ} G. Schimperna, S. Zelik, {\itshape Existence of solutions and separation from singularities for a class of fourth order degenerate parabolic equations}, Trans. Amer. Math. Soc. \textbf{365} (2013), 3799-3829.

\bibitem{S} V.N. Starovoitov, {\itshape The dynamics of a
    two-component fluid in the presence of capillary forces},
    Math. Notes \textbf{62} (1997), 244-254.

\bibitem{T} R. Temam, Navier-Stokes Equations. Theory and Numerical Analysis, North-Holland (Third edition), Oxford, 1984.

\bibitem{ZWH} L. Zhao, H. Wu, H. Huang, {\itshape Convergence to
    equilibrium for a phase-field model for the
    mixture of two viscous incompressible fluids},
    Commun. Math. Sci. \textbf{7} (2009), 939-962.

{\color{black}
\bibitem{ZF} Y. Zhou, J. Fan, {\itshape The vanishing viscosity limit for a 2D Cahn-Hilliard-Navier-Stokes
system with a slip boundary condition}, Nonlinear Anal. Real World Appl. \textbf{14} (2013), 1130-1134.}


\end{thebibliography}
\end{document}